%% file: main.tex
\documentclass[a4paper,12pt]{article} 
\usepackage{geometry}
\usepackage{amsmath, amsthm, amssymb}
\usepackage{url}

\usepackage{physics}
\usepackage{graphicx,lipsum}
\usepackage{amsfonts}
\graphicspath{ {./downloads/} }

\usepackage[colorlinks,citecolor=red,urlcolor=blue,bookmarks=false,hypertexnames=true]{hyperref}
\usepackage{tikz}
\usetikzlibrary{calc}
\usetikzlibrary{shapes}
\usepackage[autostyle]{csquotes}
\makeatletter

\theoremstyle{definition}

\newtheorem{theorem}{Theorem}[subsection]
\newtheorem{corollary}[theorem]{Corollary}
\newtheorem{lemma}[theorem]{Lemma}
\newtheorem{proposition}[theorem]{Proposition}
\newtheorem{remark}[theorem]{Remark}

\newtheorem{definition}[theorem]{Definition}

\newtheorem{fact}[theorem]{Fact}

\newcommand{\st}{\text{st}}
\newcommand{\acl}{\text{acl}}
\newcommand{\dcl}{\text{dcl}}

\newcommand{\RCF}{\text{RCF}}

\newcommand{\ODVF}{\text{ODVF}}
\newcommand{\CODF}{\text{CODF}}
\newcommand{\CODVF}{\text{CODVF}}

\newcommand{\tp}{\text{tp}}
\newcommand{\Th}{\text{Th}}

\newcommand{\cl}{\text{cl}}

\newcommand{\Jet}{\text{Jet}}

\newcommand{\rk}{\text{rk}}

\usepackage{titlesec}

\usepackage{mathtools}

\DeclarePairedDelimiterX{\inp}[2]{\langle}{\rangle}{#1, #2}
\titleformat{\chapter}
  {\Large\bfseries} 
  {}                
  {0pt}            
  {\huge}

\begin{document}
\begin{center}
\fontsize{13pt}{10pt}\selectfont
    \textsc{\textbf{T-Convexity, Tame Extensions and Definability of Hausdorff Limits in O-minimal Structures with Generic Derivations}}
    \end{center}
\vspace{0.1cm}
\begin{center}
   \fontsize{12pt}{10pt}\selectfont
    \textsc{Xiaoduo Wang}
\end{center}
\vspace{0.2cm}

\begin{abstract}
Let $T$ be a complete, model complete o-minimal extension of the theory of real closed fields in a fixed language $\mathcal{L}$. A $T$-convex subring is a convex subring that is closed under every $0$-definable continuous function, and a $T$-derivation is a derivation that is compatible with all $0$-definable $\mathcal{C}^{1}$-functions with open domains. The theory $T_{\text{convex}}$ extending $T$ by requiring that each model has a $T$-convex subring admits a relative quantifier elimination result and has been studied extensively by van den Dries and Lewenberg. On the other hand, the theory $T^{\delta}$ extending $T$ by requiring that each model be equipped with a $T$-derivation admits a model completion $T^{\delta}_{g}$, which has been studied thoroughly by Fornasiero and Kaplan. In this paper, we are going to combine the theories $T_{\text{convex}}$ and $T^{\delta}$ and show that the combined theory has a model completion $T^{\delta}_{\text{g,convex}}$. By adding an additional definable unary function $\st$, we will also give a relative quantifier elimination result for the theory of pairs $(\mathcal{M},\delta^{\mathcal{M}},\st^{\mathcal{M}},\mathcal{N},\delta^{\mathcal{N}},\st^{\mathcal{N}})$ such that $(\mathcal{M},\delta^{\mathcal{M}})$ is a model of $T^{\delta}_{g}$, $\st$ is the standard part map, $(\mathcal{N},\delta^{\mathcal{N}},\st^{\mathcal{N}})$ is a proper elementary substructure of $(\mathcal{M},\delta^{\mathcal{M}},\st^{\mathcal{M}})$, and $N$ is Dedekind complete in $M$; we call such pairs tame. An application of the relative quantifier elimination result is that, if $(\mathcal{M},\delta^{\mathcal{M}},\st^{\mathcal{M}},\mathcal{N},\delta^{\mathcal{N}},\st^{\mathcal{N}})$ is a tame pair of $T^{\delta}_{g}$, then $(\mathcal{N},\delta^{\mathcal{N}},\st^{\mathcal{N}})$ is stably embedded into $(\mathcal{M},\delta^{\mathcal{M}},\st^{\mathcal{M}})$. Lastly, we associate a sequence of $0$-definable metric topologies with models of $T^{\delta}_{g}$ and prove the Marker–Steinhorn Theorem for $T^{\delta}_{g}$. We also present the major application of the Marker–Steinhorn Theorem and the stable embedding property, that is the definability of Hausdorff limits. It is worth noting that a special case of the theory $T^{\delta}$ is the theory $\CODF$, and the quantifier elimination results of $\CODF_{\text{convex}}$ and $\CODF_{\text{tame}}$ were studied by Borrata.
\end{abstract}

\tableofcontents

\input{Introduction}

\input{Preliminaries_and_Notations}

\input{T-Convex_and_T-Derivation}

\input{T-Tame_and_T-Derivation}

\input{Definability_of_Hausdorff_Limits}

\input{Acknowledgement}

\bibliographystyle{plain}
\bibliography{reference}

\end{document}

%% file: Introduction.tex
\section{Introduction}

If $\mathcal{L}$ is any first-order language and $A$ is a set of parameters, we write ``$X$ is $\mathcal{L}(A)$-definable'' to mean that $X$ is defined by some $\mathcal{L}$-formula with parameters from $A$. In particular, we write ``$X$ is $\mathcal{L}(\varnothing)$-definable'' to mean that $X$ is both $\mathcal{L}$-definable and $0$-definable. The tame extensions in o-minimal structures have been extensively studied in the past. For a totally ordered abelian group structure $\mathcal{M}$, an element $c \in M$, and subsets $A,B \subseteq M$, we denote $|A|$ as the set of all absolute values of elements in $A$, that is, $|A| = \{|a|:a \in A\}$, and we write $c<A$ to mean that $c<a$ for every $a \in A$ and $A<B$ to mean that $a<b$ for every $a \in A$ and $b \in B$. We denote $c>A$ and $A>B$ likewise. 

Let $\mathcal{N} \preceq \mathcal{M}$. We say that the pair $(\mathcal{M},\mathcal{N})$ is \textbf{tame} if for every $a \in M$, either $|a|>N$, in which case we say $a$ is \textbf{infinite} with respect to $N$, or there exists some $b \in N$ such that $b$ is \textbf{infinitesimal} to $a$ (or $b-a$ is an \textbf{infinitesimal element}), that is $|b-a| < |N| \backslash \{0\}$. In the case that there exists such $b \in N$ that is infinitesimal to $a$, we say that $a$ is \textbf{$N$-bounded}, and we call $b$ the \textbf{standard part} of $a$, denoted by $\st(a) = b$. It is an easy exercise to see that the standard part of an element is necessarily unique and the function $\st:M \rightarrow N$ defined by
\[
\st(a) = 
\begin{cases}
b, & \text{if $b \in N$ and $b-a$ is an infinitesimal element} \\
0, & \text{otherwise}
\end{cases}
\]
is definable in the pair structure $(\mathcal{M},\mathcal{N})$. This function is called the \textbf{standard part map} induced on $\mathcal{M}$ by $\mathcal{N}$.

A well-known theorem that illustrates the connection between tame extension and the definability of types is the Marker-Steinhorn theorem in \cite{MarkerSteinhorn1994}. The theorem says that $\mathcal{M} \succeq \mathcal{N}$ is a tame extension if and only if for every nonnegative integer $n$ and every $n$-type $p(\overline{x}) \in S_{n}(N)$, $p$ is definable. This theorem covers all o-minimal structures (including those not expanding groups), but in this paper we only consider o-minimal structures expanding fields. After that, van den Dries and Lewenberg studied the tame extensions of o-minimal structures with the field structures in \cite{vandendriesLewenberg1995} by first proving results on the theory of $T$-convex subrings, denoted $T_{\text{convex}}$, which are models of $T$ with convex subrings that are closed under $0$-definable functions. They showed that $T_{\text{convex}}$ has relative quantifier elimination, and by identifying the convex hulls of tame elementary substructures as $T$-convex subrings, they showed that if $T$ is a complete o-minimal theory extending the theory of $\RCF$, then the theory $T_{\text{tame}}$ of all proper elementary pairs $\mathcal{N} \preceq \mathcal{M}$ of models of $T$ such that $\mathcal{M}$ is a tame extension of $\mathcal{N}$, together with a new unary function symbol $\st$ which is the standard part function, is complete and has quantifier elimination if $T$ has quantifier elimination and is universally axiomatizable.

An important consequence is the stable embedding property, which says that if $\mathcal{N} \preceq \mathcal{M}$ is a tame pair, then for any set $X \subseteq M^{n}$ that is definable in the structure $(\mathcal{M},\mathcal{N})$, its trace in $\mathcal{N}$, namely the set $X \cap N^{n}$, is definable in the structure $\mathcal{N}$. Later, van den Dries showed that Hausdorff limits of a definable family in an o-minimal structure extending the real field, form a definable family itself \cite{Lisbon2003}, and this is a geometric interpretation of the stable embedding property of tame extensions (many describe the definability of Hausdorff limits as the geometric interpretation of the Marker-Steinhorn Theorem).

Another extensively studied structure in model theory is a differential ring. Let $\mathcal{M}$ be a ring. A function $\delta:M \rightarrow M$ is called a \textbf{derivation} on $M$ if it is an additive group homomorphism and the multiplication satisfies the Leibniz rule: $\delta(xy) = y\delta x+x\delta y$. In 1978, Singer showed that the theory of real closed fields with a derivation, has a model completion which is called the theory of closed ordered differential fields, or simply $\CODF$ \cite{Singer1978}. The theory $\CODF$ has been studied extensively since then. For example, in \cite{BrihayeMichauxRivière2009}, it is shown that $\CODF$ has a well-defined cell decomposition theorem and a dimension theory, and in \cite{Point2011}, it is shown that $\CODF$ has open o-minimal core and admits elimination of imaginaries.

It is then natural to wonder if the theory of closed ordered differential fields with a convex subring and the theory of tame pairs of closed ordered differential fields are also complete theories with nice model-theoretic properties. The answers to these questions are positive, and they have been established by Borrata in her thesis \cite{Borrata2021}. She showed that the theory of ordered differential valued fields with a convex valuation ring, denoted $\ODVF$, has a model completion, and it is exactly the theory of closed ordered differential value fields, or simply $\CODVF$. She also showed that the theory of tame pairs $(\mathcal{M},\mathcal{N},\delta)$ of real closed fields, denoted $\RCF_{\text{tame}}^{\delta}$, that is, $\mathcal{N} \preceq \mathcal{M}$ is proper and tame, the structures $(\mathcal{M},\delta),(\mathcal{N},\delta|_{N})$ are differential rings, has a model completion which is denoted by her as $\CODF_{\text{tame}}$. She also proved relative quantifier elimination results for both $\CODVF$ and $\CODF_{\text{tame}}$, and she showed that the stable embedding property holds for $\CODF_{\text{tame}}$.

In recent years, Fornasiero and Kaplan studied a generalization of the theory $\CODF$ in \cite{FornasieroTerzo2024}. For a complete and model complete o-minimal theory $T$ extending the theory $\RCF$ in a language $\mathcal{L}$, they denoted by $\mathcal{L}^{\delta}$ the language $\mathcal{L} \cup \{\delta\}$, where $\delta$ is a unary function symbol, and they denoted the theory $T^{\delta}$ as the theory containing $T$ and an additional axiom schema saying that $\delta$ is \textbf{compatible} with every $\mathcal{L}(\varnothing)$-definable $\mathcal{C}^{1}$ function $f:U \rightarrow M$ with $U \subseteq M^{n}$ open, in the following sense:
\[
\delta f(\overline{u}) = \textbf{J}_{f}(\overline{u})\delta\overline{u},
\]
for each $\overline{u} \in U$, where $\textbf{J}_{f}(\overline{u})$ is the Jacobian matrix of $f$ at $\overline{u}$ (the Jacobian is computed with respect to the standard derivative on real closed fields; see Chapter 7 of \cite{vandenDries2003} for details on differentiations on o-minimal structures extending fields). It is not hard to see that a $T$-derivation is indeed a derivation (see Lemma 2.2 in \cite{FornasieroKaplan2020}). They showed that $T^{\delta}$ has a model completion $T^{\delta}_{g}$, which is $T^{\delta}$ with extra axioms of genericity. They also showed many properties of $\CODF$, such as having a relative quantifier elimination, a dimension theory, an $T$ open core and NIP, also hold in $T^{\delta}_{g}$. This raises the question of whether it is possible to generalize Borrata's results as well in a suitable sense, and whether one can make sense of the definability of Hausdorff limits as a geometric interpretation of the stable embedding property. In this paper, we show that all of the above have positive answers.

Relatedly, Kaplan and Pynn-Coates show that the theory of o-minimal structures expanding $\RCF$ with a $T$-derivation and a $T$-convex subring admits a model completion and is distal when the derivation is monotone in the sense of \cite{KaplanPynn-Coates2025}. We do not impose such a monotonicity constraint on the derivation in this paper.

Let $T$ be a complete and model complete o-minimal theory extending $\RCF$ in a language $\mathcal{L}$. In Section \ref{Section 2}, we present the necessary preliminaries. In Section \ref{Section 3}, we study the theory $T_{\text{convex}}^{\delta}$ of models of $T$ equipped with a $T$-derivation and a $T$-convex subring, and show that it has a model completion $T_{g,\text{convex}}^{\delta}$ (see Corollary \ref{Model Completeness of T Convex Delta G}). Moreover, if $T$ has quantifier elimination and is universally axiomatizable, then $T_{g,\text{convex}}^{\delta}$ also has quantifier elimination (see Theorem \ref{Quantifier Elimination of T Convex Delta G}). We also show that $T_{g,\text{convex}}^{\delta}$ is distal (see Proposition \ref{T delta g convex is distal}). In Section \ref{Section 4}, we consider the theory $T_{\text{tame}}^{\delta}$ of proper elementary tame pairs of models of $T^{\delta}$ expanded by a standard part map, and show that it also has a model completion $T_{g,\text{tame}}^{\delta}$ (see Corollary \ref{T tame G Complete and Model Complete}). Furthermore, if $T$ has quantifier elimination and is universally axiomatizable, then so does $T_{g,\text{tame}}^{\delta}$ (see Theorem \ref{T tame G QE}). As a consequence, $T_{g,\text{tame}}^{\delta}$ satisfies the stable embedding property (see Proposition \ref{Stable Embedding of T Tame delta G}) and has NIP (see Proposition \ref{T delta g tame has NIP}). Finally, in Section \ref{Section 5}, we focus on the geometry of models of $T^{\delta}_{g}$. We first introduce a sequence of $\mathcal{L}^{\delta}$-definable distance functions that approximate a natural but not definable topology, the $\delta$-topology, on models of $T^{\delta}_{g}$. We then prove the Marker–Steinhorn Theorem for $T^{\delta}_{g}$ (see Proposition \ref{Marker-Steinhorn for Derivation}). We then show that the Hausdorff limit of a sequence from a definable family is itself definable (see Theorem \ref{Hausdorff limits are definable}), and under an additional assumption on the pairs, that the collection of such Hausdorff limits forms a definable family (see Theorem \ref{Hausdorff Limits form a Definable Family}).

%% file: Preliminaries_and_Notations.tex
\section{Preliminaries} \label{Section 2}

In this section, we set up conventions, establish notation and list all necessary preliminary lemmas. We assume basic knowledge in model theory (for more details, see the introductory texts \cite{Marker2011} or \cite{TentZiegler2012}). For the rest of this paper, the theory $T$ is a complete and model complete o-minimal theory extending the theory of real closed fields, in the language $\mathcal{L}$. Let $\mathcal{M}$ be a model of $T$. A \textbf{$T$-convex subring} $V \subseteq M$ is a convex subring of $M$ such that $f(V) \subseteq V$ for every $\mathcal{L}(\varnothing)$-definable continuous function $f:M \rightarrow M$. A derivation $\delta$ on $M$ is called a \textbf{$T$-derivation} on $M$ if it is compatible (as defined in the introduction) with every $\mathcal{L}(\varnothing)$-definable $\mathcal{C}^{1}$-function $f:U \rightarrow M$ with $U \subseteq M^{n}$ open (in the Euclidean topology).

Let $\mathcal{L}_{\text{convex}}$ be the language $\mathcal{L}$ expanded by a unary relation symbol $V$, and let $\mathcal{L}^{\delta}$ be the language $\mathcal{L}$ expanded by a unary function symbol $\delta$. Let $\mathcal{L}^{\delta}_{\text{convex}} := \mathcal{L}_{\text{convex}} \cup \mathcal{L}^{\delta}$. We define $T_{\text{convex}}$ to be the $\mathcal{L}_{\text{convex}}$-theory of pairs $(\mathcal{M},V)$ where $\mathcal{M} \models T$ and $V$ is a proper $T$-convex subring. We also define $T^{\delta}$ to be the $\mathcal{L}^{\delta}$-theory of pairs $(\mathcal{M},\delta)$ where $\mathcal{M} \models T$ and $\delta$ is a $T$-derivation on $\mathcal{M}$. Finally, set $T_{\text{convex}}^{\delta} := T_{\text{convex}} \cup T^{\delta}$.

Let $\mathcal{L}_{\text{tame}}$ be the language $\mathcal{L}$ expanded by a unary predicate symbol $U$ and a unary function symbol $\st$. We define $T_{\text{tame}}$ to be the $\mathcal{L}_{\text{tame}}$-theory of proper tame elementary pairs of models of $T$; that is, an $\mathcal{L}_{\text{tame}}$-structure $(\mathcal{M},\mathcal{N},\st)$ satisfies $T_{\text{tame}}$ if $\mathcal{M},\mathcal{N} \models T$, $\mathcal{M}$ is a proper elementary tame extension of $\mathcal{N}$, and $\st:M \rightarrow N$ is the standard part function (defined in the introduction) induced on $\mathcal{M}$ by $\mathcal{N}$. We define $T_{\text{tame}}^{-}$ to be the same as $T_{\text{tame}}$, except that we also allow $\mathcal{N} = \mathcal{M}$. Set $\mathcal{L}_{\text{tame}}^{\delta} := \mathcal{L}_{\text{tame}} \cup \{\delta\}$. We define $T_{\text{tame}}^{\delta}$ to be the theory such that $(\mathcal{M},\mathcal{N},\st,\delta) \models T_{\text{tame}}^{\delta}$ if $(\mathcal{N},\delta|_{N}), (\mathcal{M},\delta) \models T^{\delta}$, and $(\mathcal{M},\mathcal{N},\st) \models T_{\text{tame}}$. Finally, we define $T_{\text{tame}}^{\delta,-}$ to be the same as $T_{\text{tame}}^{\delta}$, except that we also allow $\mathcal{N} = \mathcal{M}$.

For any language $\mathcal{L}_{1}$ theory $T_{1}$, let $(\mathcal{L}_{1})^{\text{df}}$ be the language obtained by adding a function symbol $f$ for every $\mathcal{L}_{1}(\varnothing)$-definable function $f$. Let $(T_{1})^{\text{df}}$ denotes the theory $T_{1}$ together with additional axioms stating that each such function symbol $f$ is exactly the corresponding $\mathcal{L}_{1}(\varnothing)$-definable function $f$. It is well known that o-minimal theories have definable Skolem functions; hence $T^{\text{df}}$ has quantifier elimination and is universally axiomatizable.

If $(\mathcal{M},\delta) \subseteq (\mathcal{N},\delta)$ are $\mathcal{L}^{\delta}$-structures and $A \subseteq N$, then $\mathcal{M}\langle A \rangle$ denotes the $\mathcal{L}$-structure generated by $M$ and $A$, and $\mathcal{M}\langle A \rangle_{\delta}$ denotes the $\mathcal{L}^{\delta}$-structure generated by $M$ and $A$. We let $\dcl_{\mathcal{M}}(A)$ denote the definable closure of $A$ in the $\mathcal{L}$-structure $\mathcal{M}$. It is well known in model theory that if $b \in \dcl_{\mathcal{M}}(A)$, then there exist an $\mathcal{L}(\varnothing)$-definable function $f : M^{n} \to M$ and a tuple $\overline{a} \in A^{n}$ such that $f(\overline{a}) = b$. It follows immediately that definable closure is invariant under taking elementary extensions. A set $B \subseteq M$ is said to be \textbf{$\dcl_{\mathcal{M}}$-independent over $A$} if $b \notin \dcl_{\mathcal{M}}(A \cup (B \backslash \{b\}))$ for all $b \in B$. It is well known in o-minimality that if $\mathcal{M} \models T$ and $A \subseteq M^{n}$, then $\dcl_{\mathcal{M}}(A)$ is an elementary substructure of $\mathcal{M}$, and $(\mathcal{M},\dcl_{\mathcal{M}})$ is a pregeometry. Let $\rk_{\mathcal{L}}$ denote the rank function associated with $(\mathcal{M},\dcl_{\mathcal{M}})$.

In later sections, we will use the following facts well-known to model theorists.

\begin{fact} (Quantifier Elimination Test) \label{Quantifier Elimination Test}
Let $\mathcal{L}_{1}$ be a language and $T_{1}$ be an $\mathcal{L}_{1}$-theory without finite models. Then $T_{1}$ has quantifier elimination if and only if the following holds: for any models $\mathcal{M},\mathcal{N} \models T_{1}$, if $\mathcal{U}$ is a common substructure of $\mathcal{M},\mathcal{N}$ with $U \neq M$, there exists $\alpha \in M \backslash U$, an elementary extension $\mathcal{N}_{1}$ of $\mathcal{N}$ and an embedding $f:U\langle\alpha\rangle \rightarrow N_{1}$ over $U$.
\end{fact}

\begin{fact} (Model Completion Test) \label{Model Completion Test}
Let $T_{1}$ and $(T_{1})^{*}$ be theories in the same language $\mathcal{L}_{1}$ such that $T_{1} \subseteq (T_{1})^{*}$. Then $(T_{1})^{*}$ is the model completion of $T_{1}$ and $(T_{1})^{*}$ has quantifier elimination if and only if 
\begin{itemize}
    \item[(i)] for every $\mathcal{M} \models T_{1}$, for every $\varphi_{1},...,\varphi_{n} \in (T_{1})^{*}$, there exists $\mathcal{N} \models (T_{1})^{*}$ such that $\mathcal{M} \subseteq \mathcal{N}$ and $\mathcal{N} \models \varphi_{1} \wedge \cdots \wedge \varphi_{n}$;
    \item[(ii)] for every $\mathcal{L}_{1}$-structures $\mathcal{A},\mathcal{B},\mathcal{C}$ such that $\mathcal{B} \models T_{1}$ and $\mathcal{C} \models (T_{1})^{*}$, and $A$ is a common substructure of $\mathcal{B}$ and $\mathcal{C}$, for every quantifier-free $\mathcal{L}_{1}(A)$-formula $\varphi(x)$, and every $b \in B$ such that $\mathcal{B} \models \varphi(b)$, there exists $c \in C$ such that $\mathcal{C} \models \varphi(c)$.
\end{itemize}
\end{fact}

\begin{fact} (Stone Duality Theorem) \label{Stone Duality Theorem}
Let $\mathcal{L}_{1}$ be a language, $T_{1}$ a complete $\mathcal{L}_{1}$-theory and $n \in \mathbb{N}$. Let $\Phi(\overline{x})$ be a set of $\mathcal{L}_{1}$-formulas, where $|\overline{x}| = n$, that is closed under disjunctions and conjunctions and contains $\top,\bot$ up to $T_{1}$-equivalence. Let $\psi(\overline{x})$ be an $\mathcal{L}_{1}$-formula. Then the following are equivalent:
\begin{itemize}
    \item[(i)] $\psi$ is $T_{1}$-equivalent to some formula from $\Phi$;
    \item[(ii)] for all $n$-types $p,q \in S_{n}(T_{1})$, if $\psi \in p$ and $\neg \psi \in q$, then there exists $\theta \in \Phi$ such that $\theta \in p$ and $\neg \theta \in q$.
\end{itemize}
\end{fact}

Lastly, we use classical results in o-minimality, such as the cell decomposition theorem, without further reference; see \cite{vandenDries2003} for details.

\subsection{$T$-Convexity and Tame Extensions}

In this section, we collect the necessary theorems on $T$-convex and $T$-tame theories that will be used later.

\begin{theorem} (Theorem 3.10 and Corollary 3.13 in \cite{vandendriesLewenberg1995}) \label{Quantifier Elimination of T Convex}
Suppose that $T$ has quantifier elimination and is universally axiomatizable. Then $T_{\text{convex}}$ has quantifier elimination. Without these assumptions on $T$, the theory $T_{\text{convex}}$ is complete and model complete.
\end{theorem}

\begin{theorem} (Theorem 5.9 and Corollary 5.10 in \cite{vandendriesLewenberg1995}) \label{Quantifier Elimination of T tame}
Suppose that $T$ has quantifier elimination and is universally axiomatizable. Then $T_{\text{tame}}$ has quantifier elimination. Without these assumptions on $T$, the theory $T_{\text{tame}}$ is complete and model complete.
\end{theorem}

The following lemma is needed in the first step in proving that $T_{\text{convex}}^{\delta}$ has a model completion.

\begin{lemma} (Lemma 2.6 in \cite{vandenDries1997}) \label{Lemma for Definable Functions in T convex}
Suppose that $f:M^{n} \rightarrow M$ is an $\mathcal{L}_{\text{convex}}(A)$-definable function for some set $A \subseteq M$ and $n > 0$. Then there exist a finite number of $\mathcal{L}(A)$-definable functions $\{f_{i}:M^{n} \rightarrow M \mid i=1,...,k\}$ such that for every $x \in M^{n}$, there is an $i \in \{1,...,k\}$ with $f_{i}(x) = f(x)$.
\end{lemma}

In a later section, where we proof a relative quantifier elimination result for the theory of tame pairs of $T^{\delta}_{g}$, we will use the following lemmas to construct appropriate embeddings.

\begin{lemma} (Corollary 5.5 and 5.6 in \cite{vandenDries2003}) \label{Extend Tame Pair}
Let $(\mathcal{M},\mathcal{N},\st) \models T_{\text{tame}}^{-}$, and let $\mathcal{M} \langle a \rangle$ be an extension of $\mathcal{M}$ with $|V| < a < |M \backslash V|$, where $V$ is the convex hull of $N$ in $M$. Then there always exist two ways to extend $(\mathcal{M},\mathcal{N},\st)$ to a model $(\mathcal{M}\langle a \rangle,\mathcal{N}_{a},\st_{a})$ of $T_{\text{tame}}^{-}$:
\begin{enumerate}
    \item $\mathcal{N}_{a} = \mathcal{N}\langle a \rangle$, so that $a \in \mathcal{N}_{a}$;
    \item $\mathcal{N}_{a} = \mathcal{N}$, so that $a \notin \mathcal{N}_{a}$.
\end{enumerate}
In both cases, the extension is unique.
\end{lemma}

\begin{remark}
Corollary 5.5 in \cite{vandendriesLewenberg1995} does not guarantee the existence of such an element $a$. 
However, the existence of $a$ follows easily from the compactness theorem.
\end{remark}

\begin{lemma} (Lemma~5.7 in \cite{vandendriesLewenberg1995}) \label{Substructures of T tame}
Suppose that $T$ has quantifier elimination and is universally axiomatizable. Then the universal part of $T_{\text{tame}}$ is $T_{\text{tame}}^{-}$. Moreover, if $(\mathcal{M},\mathcal{N},\st) \models T_{\text{tame}}^{-}$, then either:
\begin{enumerate}
    \item $\mathcal{M} \neq \mathcal{N}$, in which case $(\mathcal{M},\mathcal{N},\st) \models T_{\text{tame}}$;
    \item $\mathcal{M} = \mathcal{N}$, in which case for any $a > M$ in some elementary extension of $\mathcal{M}$, the $\mathcal{L}_{\text{tame}}$-structure $(\mathcal{M}\langle a \rangle,\mathcal{M},\st)$ is the unique model of $T_{\text{tame}}$ extending $(\mathcal{M},\mathcal{M},\st)$ such that for any $(\mathcal{M}_{1},\mathcal{N}_{1},\st_{1}) \models T_{\text{tame}}$ extending $(\mathcal{M},\mathcal{M},\st)$, the structure $(\mathcal{M}\langle a \rangle,\mathcal{M},\st)$ embeds into $(\mathcal{M}_{1},\mathcal{N}_{1},\st_{1})$ over $(\mathcal{M},\mathcal{M},\st)$.
\end{enumerate}
\end{lemma}

In establishing that the stable embedding properties extend to tame pairs of $T^{\delta}_{g}$, we will use the following lemma.

\begin{lemma} (Proposition 8.1 in the last section of \cite{Lisbon2003}) 
\label{Stable Embedding of T Tame Preparation}
Let $(\mathcal{M},\mathcal{N},\st)$ be a common elementary substructure of $(\mathcal{M}_{1},\mathcal{N}_{1},\st_{1})$ and $(\mathcal{M}_{2},\mathcal{N}_{2},\st_{2})$, both models of $T_{\text{tame}}$. Let $n>0$, $\overline{a} \in N_{1}^{n}$ and $\overline{b} \in N_{2}^{n}$. Suppose that $\tp^{\mathcal{M}_{1}}(\overline{a}/N) = \tp^{\mathcal{M}_{2}}(\overline{b}/N)$. Then $\tp^{(\mathcal{M}_{1},\mathcal{N}_{1},\st_{1})}(\overline{a}/M) = \tp^{(\mathcal{M}_{2},\mathcal{N}_{2},\st_{2})}(\overline{b}/M)$.
\end{lemma}

\subsection{Weakly O-minimality}

A linearly ordered structure $\mathcal{R} = (R,<,\dots)$ in a language $\mathcal{L}_{1}$ is called \textbf{weakly o-minimal} if every $\mathcal{L}_{1}(R)$-definable subset of $R$ is a finite union of convex sets in $R$. However, unlike o-minimality, weak o-minimality is not a first-order property, and there exist weakly o-minimal structures that are elementarily equivalent to structures that are not weakly o-minimal (see \cite{MacphersonMarkerSteinhorn2000}). Thus, it is necessary to define a theory $T_{1}$ to be a \textbf{weakly o-minimal theory} if all models of $T_{1}$ are weakly o-minimal. In \cite{vandendriesLewenberg1995}, van den Dries and Lewenberg noted the following consequence of the relative quantifier elimination of $T_{\text{convex}}$.

\begin{theorem} (Theorem 3.14 in \cite{vandendriesLewenberg1995}) \label{T Convex is Weakly O-minimal}
The theory $T_{\text{convex}}$ is weakly o-minimal.
\end{theorem}

Weakly o-minimality is an actively researched topic in model theory (see, e.g., \cite{MacphersonMarkerSteinhorn2000}). Therefore, it is convenient to use results and techniques from weakly o-minimal theories in our context. In the following, we record some useful constructions and lemmas that will be used later.

Just like in o-minimal theories, there is a useful notion of cells and a cell decomposition theorem for weakly o-minimal theories. Let $T_{1}$ be a weakly o-minimal theory and $\mathcal{N} \models T_{1}$. Let $Y \subseteq N^{n+1}$ be $0$-definable, and let $\pi:N^{n+1} \rightarrow N^{n}$ denote the projection map. Set $Z := \pi(Y)$. For $\overline{a} \in N^n$, define the \emph{fibre} of $Y$ at $\overline{a}$ by
\[
Y_{\overline{a}} := \{y \in N \mid (\overline{a},y) \in Y\}.
\]

Suppose that for each $\overline{a} \in Z$, the fibre $Y_{\overline{a}}$ is bounded above and has no supremum in $N$. Define the $0$-definable equivalence relation $\sim$ on $N^n$ by declaring $\overline{a} \sim \overline{b}$ if either $\overline{a},\overline{b} \in N^n \backslash Z$, or $\overline{a},\overline{b} \in Z$ and $\sup Y_{\overline{a}} = \sup Y_{\overline{b}}$. Let $\overline{Z} := Z/\sim$ and denote by $[\overline{a}]$ the equivalence class of $\overline{a}$. 

A natural $0$-definable ordering on $N \cup \overline{Z}$ is defined as follows: for $\overline{a} \in Z$ and $c \in N$, we set $[\overline{a}] < c$ if $w < c$ for all $w \in Y_{\overline{a}}$. Note that for $\overline{a} \neq \overline{b}$, there exists some $x \in N$ such that either $[\overline{a}] < x < [\overline{b}]$ or $[\overline{b}] < x < [\overline{a}]$. The set $\overline{Z}$ is called a \textbf{sort} in the Dedekind completion $\overline{N}$ of $N$. We say a function $F:N \rightarrow \overline{N}$ is \textbf{definable} if it is a definable function $F:N \rightarrow \overline{Z}$ for some sort $\overline{Z}$.

\begin{definition}
Let $\mathcal{N}$ be weakly o-minimal. A \textbf{\textit{cell}} is a subset of $N^{n}$, for $n>0$, defined as follows.
\begin{itemize}
    \item[(i)] A \textbf{\textit{$1$-cell}} is a definable convex subset of $N$.
    \item[(ii)] A set $X \subseteq N^{n+1}$ is an $(n+1)$-cell if there exists an \textbf{\textit{$n$-cell}} $Y \subseteq N^{n}$ such that either
    \begin{itemize}
        \item[(a)] $X = \Gamma(f|_Y)$ for some definable function $f: N^n \to N$, or
        \item[(b)] $X = (f_1,f_2)_Y$, where
        \[
        (f_1,f_2)_Y := \{ (\overline{y},z) \mid \overline{y} \in Y,\, f_1(\overline{y}) < z < f_2(\overline{y}) \},
        \]
        $A_1, A_2$ are sorts in $\overline{N}$, and $f_1: N^n \to A_1$, $f_2: N^n \to A_2$ are definable functions such that $f_1(\overline{y}) < f_2(\overline{y})$ for all $\overline{y} \in Y$ (note that $A_1$ is allowed to contain $\{-\infty\}$ and $A_2$ is allowed to contain $\{+\infty\}$).
    \end{itemize}
\end{itemize}
\end{definition}

\begin{theorem} (Theorem 4.6 in \cite{MacphersonMarkerSteinhorn2000}) \label{Weakly O-Minimal Cell Decomposition Theorem}
Let $\mathcal{N}$ be a model of a weakly o-minimal theory. Let $n>0$, and let $X_{1},\dots,X_{r}$ be definable subsets of $N^{n}$. Then there exists a finite partition $\mathcal{P}$ of $N^{n}$ into $n$-cells such that for each $1 \leq i \leq r$, the set $X_{i}$ is a union of cells from $\mathcal{P}$.
\end{theorem}

\begin{remark}
Although Macpherson, Marker and Steinhorn did not explicitly specify the parameters used to define the weakly o-minimal cells in the decomposition, it is clear from their proof that these parameters come from the same set of parameters used to define $X_{1},\dots,X_{r}$.
\end{remark}

For a weakly o-minimal theory $T_{1}$ in a language $\mathcal{L}_{1}$, there is also a natural notion of dimension. Let $\mathcal{M} \models T_{1}$ and let $U \subseteq M^{n}$ be $\mathcal{L}_{1}(M)$-definable. The \textbf{\textit{dimension}} of $U$, denoted $\dim_{\mathcal{L}_{1}}(U)$, is the largest $r \in \mathbb{N}$ for which there exists a projection $\pi: M^{n} \rightarrow M^{r}$ such that $\pi(U)$ has nonempty interior in $M^{r}$. It is well known in o-minimality that this topological dimension coincides with the algebraic dimension. For weakly o-minimal theories, the same holds provided that the model-theoretic algebraic closure satisfies the exchange property, so that it forms a pregeometry.

\begin{theorem} (Theorem 4.12 in \cite{MacphersonMarkerSteinhorn2000}) \label{Weakly O-minimal Theory and acl having Exchange Property Topological Dimension Coincides with Algebraic Dimension}
Let $T_{1}$ be a weakly o-minimal theory in a language $\mathcal{L}_{1}$ such that algebraic closure has the exchange property in all models of the theory. Let $\mathcal{N} \models T_{1}$. Then, for all $n>0$ and all $\mathcal{L}_{1}(N)$-definable sets $X \subseteq N^{n}$, we have 
\[
\dim_{\mathcal{L}_{1}}(X) = \rk_{\mathcal{L}_{1}}(X).
\]
\end{theorem}

\subsection{Generic Derivations on O-minimal Structures}

This section is devoted to explaining and recording the necessary facts about $T^{\delta}_{g}$. Recall that a derivation $\delta$ on $M$ is a $T$-derivation if it is compatible with every $\mathcal{L}(\varnothing)$-definable $\mathcal{C}^{1}$-function with open domain. The following lemma shows that a $T$-derivation also behaves well with $\mathcal{C}^{k}$-functions that are definable over parameters.

\begin{lemma} \cite{FornasieroKaplan2020} \label{delta on Ck-functions}
Suppose that $(\mathcal{M},\delta) \models T^{\delta}$. Let $k>0$ and let $f$ be an $\mathcal{L}(M)$-definable $\mathcal{C}^{k}$-function on an open set $U \subseteq M^{n}$. Then there exists a unique $\mathcal{L}(M)$-definable $\mathcal{C}^{k-1}$-function $f^{[\delta]}:U \rightarrow M$ such that
\[
\delta f(\overline{u}) = f^{[\delta]}(\overline{u}) + \textbf{J}_{f}(\overline{u})\delta\overline{u}
\]
for each $\overline{u} \in U$. Moreover, if $f$ is $\mathcal{L}(A)$-definable with $A \subseteq \ker(\delta)$, then $f^{[\delta]} = 0$.
\end{lemma}

We also need the following useful observation regarding the substructures of $T^{\delta}$ when replacing $T$ by $T^{\text{df}}$.

\begin{lemma}[Proof of Theorem 4.8 in \cite{FornasieroKaplan2020}] \label{Substructures of T delta}
Suppose that $T$ has quantifier elimination and is universally axiomatizable. Then $T^{\delta}$ is universal.
\end{lemma}

In later sections, we need to extend models of $T$ to models of $T^{\delta}$, and the following lemma serves this purpose.

\begin{lemma}[Lemma 2.13 in \cite{FornasieroKaplan2020}] \label{Extend T derivation}
Let $(\mathcal{M},\delta) \models T^{\delta}$ and let $\mathcal{N} \succ \mathcal{M}$. Let $A \subseteq N$ be a $\dcl_{\mathcal{M}}$-independent subset of $N$ over $M$ such that $\mathcal{N} = \mathcal{M}\langle A \rangle$. Then for any map $s:A \rightarrow N$, there exists a unique extension $\delta_{N}$ of $\delta$ to a $T$-derivation on $\mathcal{N}$ such that $\delta_{N}(a)=s(a)$ for all $a \in A$. 
\end{lemma}

Next, we introduce the model completion of $T^{\delta}$. We first define the notion of jet-space.

\begin{definition}
Let $(\mathcal{M},\delta)$ be an $\mathcal{L}^{\delta}$-structure. For a $k$-tuple $(n_{1},\dots,n_{k}) \in \mathbb{N}^{k}$ and $A \subseteq M^{k}$, the \textbf{\textit{$(n_{1},\dots,n_{k})$-jet-space}} of $A$ is the set
\[
\Jet_{(n_{1},\dots,n_{k})}^{\delta}(A) := \{(x_{1},\delta x_{1},\dots,\delta^{n_{1}} x_{1},\dots,x_{k},\delta x_{k},\dots,\delta^{n_{k}} x_{k}) \mid (x_{1},\dots,x_{k}) \in A\}.
\]
If $k=1$, then we simply write $\Jet_{n_{1}}^{\delta}(A)$ instead, and if $\delta$ is clear from the context, we will drop the superscript and write $\Jet_{(n_{1},\dots,n_{k})}(A)$ instead.
\end{definition}

Observe that 
\[
\Jet_{(n_{1},\dots,n_{k})}(M^{k}) = \Jet_{n_{1}}(M) \times \cdots \times \Jet_{n_{k}}(M).
\] 
For $k,m \in \mathbb{N}$, let $\Pi_{m}:M^{m+k} \rightarrow M^{m}$ be the projection map to the first $m$ coordinates. We are ready to define genericity.

\begin{definition}
A $T$-derivation $\delta$ on $\mathcal{M}$ is said to be \textbf{generic} if for every $n \in \mathbb{N}$ and every $\mathcal{L}(M)$-definable set $A \subseteq M^{n+1}$, if 
\[
\dim_{\mathcal{L}}(\Pi_{n}(A)) = n,
\] 
then there exists $a \in M$ such that $\Jet_{n}(a) \in A$. Let $T^{\delta}_{g}$ be the $\mathcal{L}^{\delta}$-theory extending $T^{\delta}$ by the axiom schema which asserts that $\delta$ is generic.
\end{definition}

The main theorem proven by Fornasiero and Kaplan is the following.

\begin{theorem} (Theorem 4.8 in \cite{FornasieroKaplan2020}) \label{T g delta has QE}
$T_{g}^{\delta}$ is the model completion of $T^{\delta}$, and if $T$ has quantifier elimination and is universally axiomatizable, then $T_{g}^{\delta}$ has quantifier elimination.
\end{theorem}

The following lemma concerns with extending models of $T^{\delta}$ to models of $T^{\delta}_{g}$.

\begin{lemma} (Proposition 4.3 in \cite{FornasieroKaplan2020}) \label{Extend T delta Structures to T delta G Structures}
Suppose that $\mathcal{M} \prec \mathcal{N} \models T$ and $(\mathcal{M},\delta) \models T^{\delta}$. Suppose further that $\rk_{\mathcal{L}}(N/M) = |N| \geq |T|$. Then there exists an extension of $\delta$ to a $T$-derivation on $\mathcal{N}$ such that $(\mathcal{N},\delta) \models T^{\delta}_{g}$.
\end{lemma}

Lastly, the following lemma shows that every $\mathcal{L}^{\delta}$-formula is $T^{\delta}_{g}$-equivalent to a formula of a special form.

\begin{lemma} (Lemma 4.11 in \cite{FornasieroKaplan2020}) \label{Forget the Derivation in T delta G}
For every $\mathcal{L}^{\delta}$-formula $\varphi$ (possibly with parameters), there exist some $n \in \mathbb{N}$ and some $\mathcal{L}$-formula $\Tilde{\varphi}$ such that
\[
T^{\delta}_{g} \vdash \forall \overline{x} \,[\varphi(\overline{x}) \leftrightarrow \Tilde{\varphi}(\Jet_{n}(\overline{x}))].
\]
\end{lemma}

\subsection{$G$-Metrics}

Fix an ordered group $\mathcal{G}$ and a set $Z$. In this section, we give definitions concerning $G$-valued metric spaces.

\begin{definition}
A function $d:Z \times Z \rightarrow G$ is called a \textbf{$G$-pseudometric} (or \textbf{$G$-valued pseudo-distance function}) on $Z$ if
\begin{itemize}
    \item[(i)] for all $x,y \in Z$, $d(x,y) \geq 0$ and $d(x,x) = 0$;
    \item[(ii)] $d(x,y) = d(y,x)$;
    \item[(iii)] for all $x,y,z \in Z$, $d(x,y) + d(y,z) \geq d(x,z)$.
\end{itemize}
If we further require that $d(x,y) = 0$ if and only if $x = y$, then $d$ is called a \textbf{$G$-metric} (or \textbf{$G$-valued distance function}) on $Z$.
\end{definition}

\begin{definition}
Given a $G$-metric $d$ on $Z$, the value $d(x,y)$ is called the \textbf{distance between $x$ and $y$ in the $G$-metric $d$}, and for any positive $\varepsilon \in G$, the set
\[
B_{d}(x,\varepsilon) := \{y \in Z \mid d(x,y)<\varepsilon\}
\]
is called the \textbf{open ball centered at $x$ of radius $\varepsilon$ with respect to the $G$-metric $d$} (or simply the \textbf{$\varepsilon$-ball of $x$ with respect to $d$}). If $d$ is clear from the context, we may omit it from the notation.
\end{definition}

\begin{remark}
Given a $G$-metric $d$ on $Z$, the collection of all open balls $B_{d}(x,\varepsilon)$ for $x \in Z$ and positive $\varepsilon \in G$ forms a basis for a topology on $Z$.
\end{remark}

\begin{definition}
Given a $G$-metric $d$ on $Z$, the topology generated by the collection of all open balls $B_d(x,\varepsilon)$ for $x \in Z$ and positive $\varepsilon \in G$ is called the \textbf{$G$-metric topology on $Z$ induced by $d$}. The triple $(Z,G,d)$ is called a \textbf{$G$-metric space}. If $d$ is clear from the context, we will simply say that $Z$ is a $G$-metric space.
\end{definition}

For $A \subseteq Z$ and positive $\varepsilon \in G$, the \textbf{$\varepsilon$-neighborhood} of $A$ is defined as
\[
U_{d}(A,\varepsilon) := \bigcup_{x \in A} B_d(x,\varepsilon).
\]
If $d$ is clear from the context, we may omit it from the notation.

%% file: T-Convex_and_T-Derivation.tex
\section{$T$-Convexity and $T$-Derivation} \label{Section 3}

Let $T_{g,\text{convex}}^{\delta} := T_{\text{convex}} \cup T_{g}^{\delta}$. In this section, we show that the theory $T_{\text{convex}}^{\delta}$ has a model completion, which is precisely $T_{g,\text{convex}}^{\delta}$. We also investigate some fundamental properties of the theory $T_{g,\text{convex}}^{\delta}$.

\subsection{Model Completion of $T_{\text{convex}}^{\delta}$}

In this section, let $(\mathcal{M},V) \models T_{\text{convex}}$. We follow the approach developed by Fornasiero and Terzo in \cite{FornasieroTerzo2024}. The core idea is as follows: first, we introduce an appropriate notion of dimension for definable subsets of models of $T_{\text{convex}}$; next, we identify the model-theoretic algebraic closure of $T_{\text{convex}}$ with the well-understood model-theoretic algebraic closure of $T$; finally, we add the axiom schema of genericity, which allows us to apply the model completion test.

By Theorem \ref{T Convex is Weakly O-minimal}, there is a natural topological dimension associated with $T_{\text{convex}}$. Note that this dimension behaves well with weakly o-minimal cells: a weakly o-minimal cell $C$ has topological dimension $r$ if and only if it is definably homeomorphic to an open weakly o-minimal cell in $M^{r}$.

Since $T_{\text{convex}}$ has an in-built total ordering, just like the o-minimal theory $T$, the model-theoretic algebraic closure $T_{\text{convex}}$-$\acl$ coincides with the definable closure $T_{\text{convex}}$-$\dcl$. Therefore, it is crucial to understand the $\mathcal{L}_{\text{convex}}(\varnothing)$-definable functions. While Lemma~\ref{Lemma for Definable Functions in T convex} already provides a description, we require a slightly stronger result, showing that such functions are precisely piecewise $\mathcal{L}(\varnothing)$-definable functions restricted to $\mathcal{L}_{\text{convex}}(\varnothing)$-definable weakly o-minimal cells.

\begin{lemma} \label{Definable Functions in T convex}
Let $U \subseteq M^{n}$, and suppose that $f:U \rightarrow M$ is an $\mathcal{L}_{\text{convex}}(A)$-definable function for some set $A \subseteq M$. Then there exists a finite partition $\mathcal{P}$ of $U$ into $\mathcal{L}_{\text{convex}}(A)$-definable weakly o-minimal cells $\{U_{i} \mid i=1,\dots,k\}$ and a finite collection of $\mathcal{L}(A)$-definable functions $\{f_{i}:M^{n} \rightarrow M \mid i=1,\dots,k\}$ such that, for each $i=1,\dots,k$, we have
\[
f|_{U_{i}} = f_{i}|_{U_{i}}.
\]
\end{lemma}

\begin{proof}
First, apply Lemma \ref{Lemma for Definable Functions in T convex} to obtain a finite collection of $\mathcal{L}(A)$-definable functions 
\[
\{f_{i}:M^{n} \rightarrow M \mid i=1,\dots,k\}
\] 
such that for every $x \in U$, there exists $i \in \{1,\dots,k\}$ with $f_{i}(x) = f(x)$. Next, for each $i = 1,\dots,k$, define 
\[
V_{i} := \{x \in U \mid f_{i}(x) = f(x)\},
\] 
which is $\mathcal{L}_{\text{convex}}(A)$-definable. Applying Theorem \ref{Weakly O-Minimal Cell Decomposition Theorem} to the collection $\{V_{i} \mid i=1,\dots,k\}$ yields a partition $\mathcal{P}$ of $U$ into weakly o-minimal cells. By restricting each $f_{i}$ to the cells in $\mathcal{P}$, the lemma follows.
\end{proof}

Sometimes we wish the functions $f_{i}$ obtained in Lemma~\ref{Definable Functions in T convex} to be of class $\mathcal{C}^{k}$, for some $k>0$. By appealing to the classical smooth cell decomposition theorem from o-minimality, we can strengthen the result as follows.

\begin{corollary}\label{Smooth Decomposition of Definable Functions in T convex}
Let $r\in\mathbb N_{>0}$, let $U\subseteq M^n$, and let $f:U\to M$ be an $\mathcal L_{\mathrm{convex}}(A)$-definable function for some $A\subseteq M$. Then there exist $\ell\in\mathbb N$, a finite partition $\mathcal P=\{U_i:i=1,\dots,\ell\}$ of $U$ into $\mathcal L_{\mathrm{convex}}(A)$-definable weakly o-minimal cells, and $\mathcal L(A)$-definable functions $f_i:M^n\to M$ for $i=1,\dots,\ell$, such that:
\begin{itemize}
\item[(a)] $f|_{U_i}=f_i|_{U_i}$ for each $i$;
\item[(b)] there exists an $\mathcal L(A)$-definable open set $V_i\supseteq U_i$ such that $f_i$ is of class $\mathcal C^r$ on $V_i$.
\end{itemize}
\end{corollary}

\begin{proof}
By Lemma \ref{Definable Functions in T convex}, we may assume that $U$ is a weakly o-minimal cell and that there exists an $\mathcal{L}(A)$-definable function $g:M^{n} \rightarrow M$ such that $g|_{U} = f$. Let $\mathcal{P}$ be an o-minimal $\mathcal{C}^{r}$-cell decomposition of $M^{n}$ such that, for each $D \in \mathcal{P}$, the restriction $g|_{D}$ is of class $\mathcal{C}^{r}$. The corollary then follows by taking a weakly o-minimal cell decomposition of $M^{n}$ that is compatible with $\mathcal{P} \cup \{U\}$.
\end{proof}

We are now ready to describe $T_{\text{convex}}$-$\dcl$.

\begin{corollary} \label{Definable Closure in T convex}
The definable closures of $T$ and $T_{\text{convex}}$ coincide. More precisely, we have
\[
T\text{-}\dcl = T_{\text{convex}}\text{-}\dcl.
\]
In particular, the definable closure in $T_{\text{convex}}$ satisfies the exchange property.
\end{corollary}

\begin{proof}
Clearly, $T$-$\dcl \subseteq T_{\text{convex}}$-$\dcl$. Let $A \subseteq M$ and let $b \in T_{\text{convex}}$-$\dcl(A)$. Then there exist $n \in \mathbb{N}$, $\overline{a} \in A$, and an $\mathcal{L}_{\text{convex}}(\varnothing)$-definable function $f:M^{n} \rightarrow M$ such that $f(\overline{a}) = b$. By Lemma \ref{Definable Functions in T convex}, there exists an $\mathcal{L}(\varnothing)$-definable function $g:M^{n} \rightarrow M$ and an $\mathcal{L}_{\text{convex}}(\varnothing)$-definable set $U$ containing $\overline{a}$ such that $f|_{U} = g|_{U}$. Hence, $g(\overline{a}) = f(\overline{a}) = b$, which shows that $b \in T$-$\dcl(A)$. The last statement follows from the fact that $(\mathcal{M},\dcl_{\mathcal{M}})$ forms a pregeometry for any o-minimal structure $\mathcal{M}$ that expands an ordered abelian group.
\end{proof}

\begin{corollary} \label{T convex Topological Dimension Coincides with Algebraic Dimension}
The topological dimension of any $\mathcal{L}_{\text{convex}}$-definable set (possibly with parameters) coincides with its algebraic dimension.
\end{corollary}

\begin{proof}
This follows from Theorem \ref{Weakly O-minimal Theory and acl having Exchange Property Topological Dimension Coincides with Algebraic Dimension} and Corollary \ref{Definable Closure in T convex}.
\end{proof}

Instead of a single axiom, we consider two axioms of genericity. We show that adding either one to the theory $T^{\delta}_{\text{convex}}$ yields a model completion of $T^{\delta}_{\text{convex}}$, and by the uniqueness of the model companion, the resulting theories are equivalent.

\begin{definition}
Let $X \subseteq M^{n}$ be an $\mathcal{L}_{\text{convex}}$-definable set (possibly with parameters). We say that $X$ is \textbf{\textit{large}} if $\dim(X) = n$.
\end{definition}

Recall that $\Pi_{k}: M^{n} \to M^{k}$ denotes the projection onto the first $k$ coordinates. We now define the same axiom schemas as those appearing in \cite{FornasieroTerzo2024}.

\begin{itemize}
    \item[(i)] (\textbf{Deep}) For every $\mathcal{L}_{\text{convex}}(M)$-definable set $X \subseteq M^{n+1}$, if $\Pi_{n}(X)$ is large, then there exists $a \in M$ such that $\Jet_{n}(a) \in X$.
    \item[(ii)] (\textbf{Wide}) For every $\mathcal{L}_{\text{convex}}(M)$-definable set $Y \subseteq M^{n} \times M^{n}$, if $\Pi_{n}(Y)$ is large, then there exists $\overline{b} \in M^{n}$ such that $(\overline{b}, \delta \overline{b}) \in Y$.
\end{itemize}

Let $T^{\delta}_{\text{deep,convex}} = T^{\delta}_{\text{convex}} \cup \text{(Deep)}$ and $T^{\delta}_{\text{wide,convex}} = T^{\delta}_{\text{convex}} \cup \text{(Wide)}$.

\begin{lemma} (Essentially Lemma 3.13 in \cite{FornasieroTerzo2024}) \label{T Convex Delta QE Test Preparation i}
We have
\[
T^{\delta}_{\text{wide,convex}} \vdash T^{\delta}_{\text{deep,convex}}.
\]
\end{lemma}

\begin{proof}
Fix an $\mathcal{L}_{\text{convex}}(M)$-definable set $X \subseteq M^{n+1}$ such that $\Pi_{n}(X)$ is large. Define
\[
Y := \left\{ (\overline{x},\overline{y}) \in M^{n} \times M^{n} \;\Bigg|\; (\overline{x},y_{n}) \in X \;\wedge\; \bigwedge_{i=1}^{n-1} y_{i} = x_{i+1} \right\}.
\]
By construction, we have $\Pi_{n}(Y) = \Pi_{n}(X)$, and hence $\Pi_{n}(Y)$ is large. By (Wide), there exists $\overline{b} = (b_{1},\dots,b_{n}) \in M^{n}$ such that $(\overline{b},\delta \overline{b}) \in Y$. By the definition of $Y$, it follows that $\Jet_{n}(b_{1}) \in X$.
\end{proof}

\begin{lemma} \label{T Convex Delta QE Test Preparation ii}
Let $(\mathcal{M},V,\delta) \models T_{\text{convex}}^{\delta}$. Let $X \subseteq M^{n} \times M^{n}$ be $\mathcal{L}_{\text{convex}}(M)$-definable such that $\Pi_{n}(X)$ is large. Then there exists an $\mathcal{L}_{\text{convex}}^{\delta}$-structure $(\mathcal{N},V_{N},\varepsilon) \supseteq (\mathcal{M},V,\delta)$ and an $n$-tuple $\overline{b} \in N^{n}$ such that 
\begin{itemize}
    \item[(a)] $(\mathcal{N},V_{N}) \succeq (\mathcal{M},V)$,
    \item[(b)] $(\mathcal{N},V_{N},\varepsilon) \models T_{\text{convex}}^{\delta}$, and
    \item[(c)] $(\overline{b},\varepsilon \overline{b}) \in X^{(\mathcal{N},V_{N})}$,
\end{itemize}
where $X^{(\mathcal{N},V_{N})} \subseteq N^{n} \times N^{n}$ is defined by the same $\mathcal{L}_{\text{convex}}(M)$-formula that defines $X$.
\end{lemma}

\begin{proof}
Let $(\mathcal{N},V_{N}) \succeq (\mathcal{M},V)$ be $|M|^{+}$-saturated. By Corollaries \ref{Definable Closure in T convex} and \ref{T convex Topological Dimension Coincides with Algebraic Dimension}, there exists $\overline{b} \in \Pi_{n}\bigl(X^{(\mathcal{N},V_{N})}\bigr)$ that is $\dcl_{\mathcal{M}}$-independent over $M$. Let $\overline{d} \in N^{n}$ be such that $(\overline{b},\overline{d}) \in X^{(\mathcal{N},V_{N})}$. By Lemma \ref{Extend T derivation}, there exists a $T$-derivation $\varepsilon$ on $\mathcal{N}$ extending $\delta$ such that $\varepsilon(\overline{b}) = \overline{d}$.
\end{proof}

\begin{lemma} \label{Forget the Derivation in T convex}
For every quantifier-free $\mathcal{L}_{\text{convex}}^{\delta}$-formula $\varphi$ (possibly with parameters), there exist $n \in \mathbb{N}$ and a quantifier-free $\mathcal{L}_{\text{convex}}$-formula $\Tilde{\varphi}$ such that
\[
T_{\text{convex}}^{\delta} \models \forall \overline{x} \, [\varphi(\overline{x}) \leftrightarrow \Tilde{\varphi}(\Jet_{n}(\overline{x}))].
\]
\end{lemma}

\begin{proof}
Let $e(\varphi)$ be the total number of times such that the function $\delta$ is applied to a term in $\varphi$ that is not of the form $\delta^{k} x_{i}$. For example, if $\varphi(x)$ is the $\mathcal{L}^{\delta}_{\text{convex}}$-formula ``$\delta(t_{1}(x)) = 0 \wedge \delta^{2}(t_{2}(x)) < 0 \vee \delta^{5}x=0$'', where $t_{1},t_{2}$ are $\mathcal{L}_{\text{convex}}$-terms that are not variables, then $e(\varphi) = 3$. We prove the lemma by induction on $e(\varphi)$. The base case $e(\varphi) = 0$ is trivial since we can take $\Tilde{\varphi}$ by replacing every occurrence of $\delta^{k} x_{i}$ in $\varphi$ with a new variable $x_{i}^{(k)}$. Now suppose that $e(\varphi) > 0$. Then there exist $m \in \mathbb{N}$, a quantifier-free $\mathcal{L}_{\text{convex}}^{\delta}$-formula $\psi$ (possibly with parameters), and a quantifier-free $\mathcal{L}_{\text{convex}}(\varnothing)$-definable function $f$ such that
\[
\varphi(\overline{x}) = \psi\bigl(\overline{x}, \delta f(\Jet_{m}(\overline{x}))\bigr).
\]
By Corollary \ref{Smooth Decomposition of Definable Functions in T convex}, there exists a weakly o-minimal cell decomposition $\mathcal{P}$ such that for each $D \in \mathcal{P}$, there is an $\mathcal{L}(\varnothing)$-definable $\mathcal{C}^{1}$-function $f_{D}$ on an open $\mathcal{L}(\varnothing)$-definable set $U_{D}$ with $f_{D}|_{D} = f|_{D}$. By the definition of $T$-derivation, in any model of $T_{\text{convex}}^{\delta}$, for each $\overline{y} \in D$, we have
\[
\delta f(\overline{y}) = \textbf{J}_{f_{D}}(\overline{y}) \, \delta \overline{y}.
\]
Set $\Tilde{f}$ to be the piecewise $\mathcal L(\varnothing)$-definable matrix-valued function such that $\Tilde{f}(\overline{y})=\mathbf{J}_{f_D}(\overline{y})$ whenever $\overline{y}\in D$. Define 
\[
\theta(\overline{x}) := \psi\bigl(\overline{x}, \Tilde{f}(\Jet_{m}(\overline{x})) \, \delta(\Jet_{m}(\overline{x}))\bigr).
\] 
Then, we have $e(\theta) < e(\varphi)$ and
\[
T_{\text{convex}}^{\delta} \models \forall \overline{x} \, [\varphi(\overline{x}) \leftrightarrow \theta(\overline{x})].
\]
By the inductive hypothesis, the lemma follows.
\end{proof}

\begin{lemma} \label{T Convex Delta QE Test Preparation iii}
Let $(\mathcal{B},V_{B},\delta_{B}) \models T^{\delta}_{\text{convex}}$ and $(\mathcal{C},V_{C},\delta_{C}) \models T^{\delta}_{\text{deep,convex}}$. Let $(\mathcal{A},V_{A},\delta_{A})$ be a common $\mathcal{L}^{\delta}_{\text{convex}}$-substructure of $(\mathcal{B},V_{B},\delta_{B})$ and $(\mathcal{C},V_{C},\delta_{C})$. Suppose further that $(\mathcal{B},V_{B})$ and $(\mathcal{C},V_{C})$ have the same $\mathcal{L}_{\text{convex}}(A)$-theory. Let $\theta(x)$ be a quantifier-free $\mathcal{L}_{\text{convex}}^{\delta}(A)$-formula, and let $b \in B$ satisfy 
\[
(\mathcal{B},V_{B},\delta_{B}) \models \theta(b).
\]
Then there exists $c \in C$ such that 
\[
(\mathcal{C},V_{C},\delta_{C}) \models \theta(c).
\]
\end{lemma}

\begin{proof}
By Lemma \ref{Forget the Derivation in T convex}, there exist $n \in \mathbb{N}$ and an $\mathcal{L}_{\text{convex}}(A)$-formula $\psi$ such that 
\[
\theta(x) = \psi(\Jet_{n}(x)).
\]
Define 
\[
Y^{(\mathcal{B},V_{B})} := \{\overline{d} \in B^{n+1} \mid (\mathcal{B},V_{B}) \models \psi(\overline{d})\},
\]
\[
Y^{(\mathcal{C},V_{C})} := \{\overline{d} \in C^{n+1} \mid (\mathcal{C},V_{C}) \models \psi(\overline{d})\}.
\]
Let $d$ be the least integer such that $\delta_{B}^{d}(b)$ is $\dcl_{\mathcal{B}}$-dependent over $\Jet^{\delta_{B}}_{d-1}(b) \cup A$, and set $d = +\infty$ if no such integer exists.  

Case 1: $d \geq n$. Since $\Jet_{n-1}^{\delta_{B}}(b) \in \Pi_{n}\big(Y^{(\mathcal{B},V_{B})}\big)$, we obtain
\[
\dim\!\big(\Pi_{n}(Y^{(\mathcal{C},V_{C})})\big) 
= \dim\!\big(\Pi_{n}(Y^{(\mathcal{B},V_{B})})\big) 
= \rk\!\big(\Pi_{n}(Y^{(\mathcal{B},V_{B})})\big) 
= n.
\]
Thus $\Pi_{n}(Y^{(\mathcal{C},V_{C})})$ is large, and hence, by (Deep), there exists $c \in C$ such that 
\[
(\mathcal{C},V_{C}) \models \psi(\Jet^{\delta_{C}}_{n}(c)).
\]

Case 2: $d < n$. Then $\delta_{B}^{d}b \in \dcl_{\mathcal{B}}(\Jet^{\delta_{B}}_{d-1}(b) \cup A)$. Hence there exists an $\mathcal{L}(A)$-definable function $f \colon M^{d} \to M$ such that
\[
\delta_{B}^{d}b = f(b,\dots,\delta_{B}^{d-1}b).
\]
By o-minimal cell decomposition, we may assume there is an open $\mathcal{L}(A)$-definable set $V$ such that $f|_{V}$ is of class $\mathcal{C}^{\,n-d}$. By Lemma \ref{delta on Ck-functions} (parameters can be controlled, although this is not explicitly stated in \cite{FornasieroKaplan2020}), there exist $\mathcal{L}(A)$-definable functions 
\[
f_{d+1},\dots,f_{n} \colon V \to M
\]
such that for each $d+1 \leq i \leq n$, we have
\[
\delta_{B}^{i}b = f_{i}(\Jet^{\delta_{B}}_{d-1}(b)).
\]
For $\overline{y}=(y_0,\dots,y_d)$, write $\overline{y}_{<d}:=(y_0,\dots,y_{d-1})$. Define
\[
\begin{aligned}
Z^{(\mathcal{B},V_B)}
:=\bigl\{\overline{y}\in B^{d+1}:{}\ &
\overline{y}_{<d}\in V,\quad
f(\overline{y}_{<d})=y_d,\\
&(\mathcal{B},V_B)\models
\psi\bigl(\overline{y},
f_{d+1}(\overline{y}_{<d}),\dots,
f_n(\overline{y}_{<d})\bigr)
\bigr\}.
\end{aligned}
\]
Since $\Jet_{d-1}^{\delta_B}(b)\in\Pi_d\bigl(Z^{(\mathcal B,V_B)}\bigr)$, the set $\Pi_{d}\big(Z^{(\mathcal{C},V_{C})}\big)$ is large. Hence, by (Deep), there exists $c \in C$ such that 
\[
\Jet^{\delta_{C}}_{d}(c) \in Z^{(\mathcal{C},V_{C})},
\]
which implies
\[
(\mathcal{C},V_{C}) \models \psi(\Jet^{\delta_{C}}_{n}(c)).
\]
\end{proof}

\begin{theorem} \label{Quantifier Elimination of T Convex Delta G}
Suppose that $T$ has quantifier elimination and is universally axiomatizable. Then $T_{\text{deep,convex}}^{\delta} = T_{\text{wide,convex}}^{\delta}$, and this common theory is the model completion $T_{g,\text{convex}}^{\delta}$ of $T_{\text{convex}}^{\delta}$. Moreover, the theory $T_{g,\text{convex}}^{\delta}$ has quantifier elimination.
\end{theorem}

\begin{proof}
By Theorem \ref{Quantifier Elimination of T Convex}, the theory $T_{\text{convex}}$ has quantifier elimination. To prove the statement, we apply the model completion test (Fact \ref{Model Completion Test}). The conditions required for this test are verified by combining Lemmas \ref{T Convex Delta QE Test Preparation i}, \ref{T Convex Delta QE Test Preparation ii}, and \ref{T Convex Delta QE Test Preparation iii}.
\end{proof}

\begin{corollary} \label{Model Completeness of T Convex Delta G}
The theory $T_{g,\text{convex}}^{\delta}$ is complete and model complete (recall that $T$ is assumed to be complete and model complete).
\end{corollary}

\begin{proof}
For completeness, we may replace $\mathcal{L}$ by $\mathcal{L}^{\text{df}}$ and $T$ by $T^{\text{df}}$, so that Theorem \ref{Quantifier Elimination of T Convex Delta G} applies. Let $\mathcal{P} \models T$ be the prime model. Then the $\mathcal{L}_{\text{convex}}^{\delta}$-structure $(\mathcal{P},P,\delta_{P})$, where $\delta_{P}$ is the trivial derivation on $P$, embeds into every model of $T_{g,\text{convex}}^{\delta}$. This shows that $T_{g,\text{convex}}^{\delta}$ is complete.

For model completeness, note that $T$ is model complete and $(T^{\text{df}})_{g,\text{convex}}^{\delta}$ has quantifier elimination. Hence every $\mathcal{L}_{\text{convex}}^{\delta}$-formula is $T_{g,\text{convex}}^{\delta}$-equivalent to an existential $\mathcal{L}_{\text{convex}}^{\delta}$-formula, as required.
\end{proof}

\begin{corollary} \label{Forget the Derivation in T convex delta G}
For every $\mathcal{L}_{\text{convex}}^{\delta}$-formula $\varphi$ (possibly with parameters), there exist $n \in \mathbb{N}$ and an $\mathcal{L}_{\text{convex}}$-formula $\Tilde{\varphi}$ such that
\[
T_{g,\text{convex}}^{\delta} \models \forall \overline{x} \, \bigl(\varphi(\overline{x}) \leftrightarrow \Tilde{\varphi}(\Jet_{n}(\overline{x}))\bigr).
\]
\end{corollary}

\begin{proof}
By Lemma \ref{Forget the Derivation in T convex}, the statement holds for quantifier-free formulas.  
By Theorem \ref{Quantifier Elimination of T Convex Delta G}, every formula is equivalent modulo $T_{g,\text{convex}}^{\delta}$ to a quantifier-free one.  
Combining these two facts yields the desired conclusion.
\end{proof}

\subsection{NIP and Distality of $T^{\delta}_{g,\text{convex}}$}

Distal theories, introduced by Simon in \cite{Simon2013}, form a subclass of NIP theories. In \cite{FornasieroKaplan2020}, Fornasiero and Kaplan showed that $T^{\delta}_{g}$ is distal. We establish an analogous result for $T^{\delta}_{g,\text{convex}}$. For a more thorough treatment of NIP theories, see \cite{Simon2015}. Let $(\mathbb{M},V,\delta)$ denote the monster model of $T^{\delta}_{g,\text{convex}}$.

\begin{definition} (See Lemma 2.7 in \cite{Simon2013}) \label{Distality Defined}
$T^{\delta}_{g,\text{convex}}$ is \textbf{distal} if whenever $\overline{b}$ is a tuple from $\mathbb{M}$ and $(\overline{a}_{i})_{i \in I}$ is an $\mathcal{L}_{\text{convex}}^{\delta}(\varnothing)$-indiscernible sequence from $\mathbb{M}$ such that $I = I_{1} + (c) + I_{2}$ where $I_{1},I_{2}$ are infinite without endpoints and $(\overline{a}_{i})_{i \in I_{1}+I_{2}}$ is $\mathcal{L}_{\text{convex}}^{\delta}(\overline{b})$-indiscernible, then $(\overline{a}_{i})_{i \in I}$ is $\mathcal{L}_{\text{convex}}^{\delta}(\overline{b})$-indiscernible.
\end{definition}

\begin{lemma} (Lemma 7.6 (ii) in \cite{MoconjaTanović2025}) \label{Weakly O-Minimal Theories are Distal}
Weakly o-minimal theories are distal.
\end{lemma}

\begin{remark}
Lemma~7.6~(ii) in \cite{MoconjaTanović2025} actually establishes a stronger result for weakly o-minimal pairs. 
Note that in a weakly o-minimal theory, every global type paired with the in-built ordering is a weakly o-minimal type.
\end{remark}

\begin{proposition} \label{T delta g convex is distal}
The theory $T^{\delta}_{g,\text{convex}}$ is distal.
\end{proposition}

\begin{proof}
We adopt the notation from Definition \ref{Distality Defined}. Let \(\varphi(\overline{x}_{1},...,\overline{x}_{n},\overline{y})\) be an \(\mathcal{L}_{\mathrm{convex}}^{\delta}(\varnothing)\)-formula. It suffices to show that
\[
(\mathbb{M},V,\delta) \models \varphi(\overline{a}_{i_{1}},...,\overline{a}_{i_{n}},\overline{b}) \leftrightarrow \varphi(\overline{a}_{j_{1}},...,\overline{a}_{j_{n}},\overline{b})
\]
for any \(i_{1}<\cdots<i_{n}\) and \(j_{1}<\cdots<j_{n}\) in \(I\). By Corollary \ref{Forget the Derivation in T convex delta G}, there exist \(m \in \mathbb{N}\) and an
\(\mathcal{L}_{\mathrm{convex}}(\varnothing)\)-formula \(\Tilde{\varphi}\) such that
\[
T^{\delta}_{g,\text{convex}} \vdash \forall \overline{x}_{1} \cdots \forall \overline{x}_{n} \bigl[\varphi(\overline{x}_{1},...,\overline{x}_{n},\overline{y}) \leftrightarrow \Tilde{\varphi}(\Jet_{m}(\overline{x}_{1}),...,\Jet_{m}(\overline{x}_{n}),\Jet_{m}(\overline{y}))\bigr].
\]
Since \((\overline{a}_{i})_{i \in I}\) is \(\mathcal{L}_{\text{convex}}^{\delta}(\varnothing)\)-indiscernible, we have that \((\Jet_{m}(\overline{a}_{i}))_{i \in I}\) is \(\mathcal{L}_{\text{convex}}^{\delta}(\varnothing)\)-indiscernible and in particular \(\mathcal{L}_{\text{convex}}(\varnothing)\)-indiscernible. Similarly, since \((\overline{a}_{i})_{i \in I_{1}+I_{2}}\) is \(\mathcal{L}_{\text{convex}}^{\delta}(\overline{b})\)-indiscernible, we have that \((\Jet_{m}(\overline{a}_{i}))_{i \in I_{1}+I_{2}}\) is \(\mathcal{L}_{\text{convex}}(\Jet_{m}(\overline{b}))\)-indiscernible. By Lemma \ref{Weakly O-Minimal Theories are Distal}, we have
\[
(\mathbb{M},V) \models \Tilde{\varphi}(\Jet_{m}(\overline{a}_{i_{1}}),...,\Jet_{m}(\overline{a}_{i_{n}}),\Jet_{m}(\overline{b})) \leftrightarrow \Tilde{\varphi}(\Jet_{m}(\overline{a}_{j_{1}}),...,\Jet_{m}(\overline{a}_{j_{n}}),\Jet_{m}(\overline{b})).
\]
Applying Corollary \ref{Forget the Derivation in T convex delta G} again, the desired result follows.
\end{proof}

\begin{corollary}
The theory $T^{\delta}_{g,\text{convex}}$ has NIP.
\end{corollary}

\begin{remark}
Proposition~7.1 in \cite{AschenbrennerChernikovGehretZiegler2022} provides a useful criterion for checking distality. However, to apply this proposition in our setting, one needs to modify the language by replacing the unary predicate symbol $U$, which denotes the valuation ring, with a unary function symbol $\nu$ denoting the valuation.
\end{remark}

%% file: T-Tame_and_T-Derivation.tex
\section{The Theory of Tame Pairs} \label{Section 4}

Let $T_{g,\text{tame}}^{\delta}$ be the theory such that $(\mathcal{M},\mathcal{N},\st,\delta) \models T_{g,\text{tame}}^{\delta}$ if $(\mathcal{N},\delta|_{\mathcal{N}}), (\mathcal{M},\delta) \models T^{\delta}_{g}$ and $(\mathcal{M},\mathcal{N},\st) \models T_{\text{tame}}$. In this section, we show that $T_{g,\text{tame}}^{\delta}$ is the model completion of $T_{\text{tame}}^{\delta}$ and admits a relative quantifier elimination. We also prove that $T_{g,\text{tame}}^{\delta}$ has the stable embedding property.

\subsection{Model Completion of $T_{\text{tame}}^{\delta}$}

The strategy is to use the quantifier elimination test \ref{Quantifier Elimination Test}. We will need to understand the substructures of models of $T_{g,\text{tame}}^{\delta}$ well. The core purpose of adding the unary function symbol $\st$ to our language is to make these substructures behave nicely.

\begin{lemma} \label{Universal Part of T tame delta and Extension to Models of T tame G}
Suppose that $T$ has quantifier elimination and is universally axiomatizable. Then the universal part of $T_{\text{tame}}^{\delta}$ is $T_{\text{tame}}^{\delta,-}$. Furthermore, every $(\mathcal{M},\mathcal{N},\st,\delta) \models T_{\text{tame}}^{\delta,-}$ can be extended to a model of $T_{g,\text{tame}}^{\delta}$.
\end{lemma}

\begin{proof}
For the first part, it suffices to show that $(\mathcal{M},\mathcal{N},\st,\delta) \models T_{\text{tame}}^{\delta,-}$ if and only if $(\mathcal{M},\mathcal{N},\st,\delta)$ is a substructure of a model of $T_{\text{tame}}^{\delta}$.

``$\Rightarrow$'': Fix $(\mathcal{M},\mathcal{N},\st,\delta) \models T_{\text{tame}}^{\delta,-}$. By Lemma \ref{Substructures of T tame}, there exists an embedding $(\mathcal{M},\mathcal{N},\st) \subseteq (\mathcal{M}_{1},\mathcal{N}_{1},\st_{1}) \models T_{\text{tame}}$, where $\mathcal{N} = \mathcal{N}_{1}$, and either $\mathcal{M}_{1} = \mathcal{M}$ or $\mathcal{M}_{1} = \mathcal{M}\langle a \rangle$. In the former case, there is nothing to show. In the latter case, since $a \notin M$, it is $\dcl_{\mathcal{M}}$-independent over $M$. By Lemma \ref{Extend T derivation}, there is a unique $T$-derivation $\delta_{a}$ on $\mathcal{M}\langle a \rangle$ extending $\delta$ such that $\delta_{a} a = 0$.

``$\Leftarrow$'': Fix a substructure $(\mathcal{M},\mathcal{N},\st,\delta)$ of some $(\mathcal{M}_{1},\mathcal{N}_{1},\st_{1},\delta_{1}) \models T_{\text{tame}}^{\delta}$. Then, clearly we have $(\mathcal{M},\mathcal{N},\st) \subseteq (\mathcal{M}_{1},\mathcal{N}_{1},\st_{1}) \models T_{\text{tame}}$, so by Lemma \ref{Substructures of T tame}, $(\mathcal{M},\mathcal{N},\st) \models T_{\text{tame}}^{-}$. Also, $(\mathcal{M},\delta) \subseteq (\mathcal{M}_{1},\delta_{1})$, and we have $(\mathcal{N},\delta|_{N}) \subseteq (\mathcal{M},\delta)$ since $T_{\text{tame}}^{\delta} \vdash \forall x [U x \rightarrow U(\delta x)]$. As $T$ is universally axiomatizable and has quantifier elimination, every $\mathcal{L}(\varnothing)$-definable function is piecewise given by $\mathcal{L}(\varnothing)$-terms. Hence $T^{\delta}$ is universal, so $(\mathcal{N},\delta|_{N}), (\mathcal{M},\delta) \models T^{\delta}$.

For the second part, we show that $(\mathcal{M},\mathcal{N},\st,\delta) \models T_{\text{tame}}^{\delta,-}$ can always be extended to a model $(\mathcal{M}_{1},\mathcal{N}_{1},\st_{1},\delta_{1}) \models T_{\text{tame}}^{\delta,-}$ such that $(\mathcal{N}_{1},\delta_{1}|_{N_{1}}) \models T^{\delta}_{g}$. Let $a \notin M$ in the monster model $\mathbb{M}$ (as an $\mathcal{L}$-structure) satisfy $|V| < a < |M \setminus V|$, where $V$ is the convex hull of $N$ in $M$. By Lemma \ref{Extend Tame Pair}, $(\mathcal{M}\langle a \rangle, \mathcal{N}\langle a \rangle, \st_{a})$ properly extends $(\mathcal{M},\mathcal{N},\st)$. By transfinite induction, we can construct a set $A \subseteq \mathbb{M} \setminus M$ such that $|A| = |M| \geq |T|$, $A$ is $\dcl_{\mathcal{M}}$-independent over $M$ and $(\mathcal{M}\langle A \rangle, \mathcal{N}\langle A \rangle, \st_{A})$ extends $(\mathcal{M},\mathcal{N},\st)$. By Lemma \ref{Extend T delta Structures to T delta G Structures}, there is a $T$-derivation $\delta_{A}$ on $\mathcal{N}\langle A \rangle$ extending $\delta|_{N}$ such that $(\mathcal{N}\langle A \rangle, \delta_{A}) \models T^{\delta}_{g}$. Note that $\delta_{A}|_{A}:A \rightarrow \mathcal{M}\langle A \rangle$ is a map, the derivation $\delta$ is a $T$-derivation on $\mathcal{M}$, and $A$ is by construction $\dcl_{\mathcal{N}}$-independent over $N$. Then by Lemma \ref{Extend T derivation}, there is a unique $T$-derivation $\delta^{*}$ on $\mathcal{M}\langle A \rangle$ extending $\delta$ such that $\delta^{*}(a) = \delta_{A}|_{A}(a)$ for each $a \in A$. Furthermore, if $x \in N\langle A \rangle$, then there are $k_{1},k_{2} \in \mathbb{N}$, an $k_{1}+k_{2}$-nary $\mathcal{L}(\varnothing)$-definable function $g$ and $\overline{a} \in A^{k_{1}},\overline{y} \in N^{k_{2}}$ such that $x = g(\overline{a},\overline{y})$. By o-minimality, we may assume that $g$ is continuously differentiable at $(\overline{a},\overline{y})$. Then
\[
\delta^{*} x = \delta^{*} g(\overline{a},\overline{y}) = \sum_{i=1}^{k_{1}}\frac{\partial g}{\partial a_{i}}(\overline{a})\delta^{*} a_{i} + \sum_{i=1}^{k_{2}}\frac{\partial g}{\partial y_{i}}(\overline{y})\delta^{*} y_{i}
\]
\[
= \sum_{i=1}^{k_{1}}\frac{\partial g}{\partial a_{i}}(\overline{a})\delta_{A} a_{i} + \sum_{i=1}^{k_{2}}\frac{\partial g}{\partial y_{i}}(\overline{y})\delta y_{i} \in N\langle A \rangle,
\]
as $\delta|_{N}$ is a $T$-derivation on $N$. Thus $\delta^{*}|_{\mathcal{N}\langle A \rangle}$ is a $T$-derivation on $\mathcal{N}\langle A \rangle$ such that $\delta^{*}|_{\mathcal{N}\langle A \rangle}(a) = \delta_{A}|_{A}(a)$ and $\delta^{*}|_{\mathcal{N}\langle A \rangle}$ coincides with $\delta_{A}$ on $\mathcal{N}\langle A \rangle$. Again, by Lemma \ref{Extend T derivation}, such $T$-derivation is unique, and thus we must have $\delta^{*}|_{\mathcal{N}\langle A \rangle}=\delta_{A}$. This shows that we have $(\mathcal{M}\langle A \rangle, \mathcal{N}\langle A \rangle,\st_{A},\delta^{*}) \models T_{\text{tame}}^{\delta,-}$ and $(\mathcal{N}\langle A \rangle,\delta^{*}|_{\mathcal{N}\langle A \rangle}) \models T^{\delta}_{g}$.

Without loss of generality and for simpler notation, we may assume that $(\mathcal{M},\mathcal{N},\st,\delta) \models T_{\text{tame}}^{\delta,-}$ and $(\mathcal{N},\delta|_{N}) \models T^{\delta}_{g}$. We now show that we can always extend $(\mathcal{M},\mathcal{N},\st,\delta) \models T_{\text{tame}}^{\delta,-}$ to a model $(\mathcal{M}_{1},\mathcal{N},\st_{1},\delta_{1}) \models T_{\text{tame}}^{\delta}$ such that $(\mathcal{M},\delta_{1}),(\mathcal{N},\delta_{1}|_{N}) \models T^{\delta}_{g}$. By Lemma \ref{Extend Tame Pair}, we may use the same techniques as in the previous paragraph. Let $A$ be a $\dcl_{\mathcal{M}}$-independent set over $M$ such that $|A| = |M| \geq |T|$, and $(\mathcal{M}\langle A \rangle, \mathcal{N},\st_{A})$ extends $(\mathcal{M}, \mathcal{N},\st)$. By Lemma \ref{Extend T delta Structures to T delta G Structures}, there is a $T$-derivation $\delta^{*}$ on $\mathcal{M}\langle A \rangle$ extending $\delta$ such that $(\mathcal{M}\langle A \rangle,\delta^{*}) \models T^{\delta}_{g}$. Clearly $\mathcal{M}\langle A \rangle \supsetneq \mathcal{N}$, and hence $(\mathcal{M},\mathcal{N},\st,\delta) \subseteq (\mathcal{M}\langle A \rangle,\mathcal{N},\st_{A},\delta^{*}) \models T^{\delta}_{g,\text{tame}}$.
\end{proof}

\begin{corollary} \label{Universal Part of T tame delta G}
Suppose that $T$ has quantifier elimination and is universally axiomatizable. Then, the universal theory of $T_{g,\text{tame}}^{\delta}$ is $T_{\text{tame}}^{\delta,-}$.
\end{corollary}

\begin{lemma} \label{Embedding Lemma for T tame delta G}
Suppose that $T$ has quantifier elimination and is universally axiomatizable. Let $(\mathcal{M}_{1},\mathcal{N}_{1},\st_{1},\delta_{1}),(\mathcal{M}_{2},\mathcal{N}_{2},\st_{2},\delta_{2}) \models T_{g,\mathrm{tame}}^{\delta}$, and let $(\mathcal{M},\mathcal{N},\st,\delta)$ be a common $\mathcal{L}_{\mathrm{tame}}^{\delta}$-substructure of both. Suppose that $M \neq M_{1}$ and that $(\mathcal{M}_{2},\mathcal{N}_{2},\st_{2},\delta_{2})$ is sufficiently saturated. Then there is an $\mathcal{L}_{\mathrm{tame}}^{\delta}$-substructure of $(\mathcal{M}_{1},\mathcal{N}_{1},\st_{1},\delta_{1})$ properly extending $(\mathcal{M},\mathcal{N},\st,\delta)$ which embeds into $(\mathcal{M}_{2},\mathcal{N}_{2},\st_{2},\delta_{2})$ over $(\mathcal{M},\mathcal{N},\st,\delta)$.
\end{lemma}

\begin{proof}
Choose $a \in M_{1} \backslash M$. Let $V,V_{1},V_{2}$ be the convex hulls of $N,N_{1},N_{2}$ in $M,M_{1},M_{2}$, respectively. Then $(\mathcal{M},V,\delta)$ is a common $\mathcal{L}_{\mathrm{convex}}^{\delta}$-substructure of $(\mathcal{M}_{1},V_{1},\delta_{1})$ and $(\mathcal{M}_{2},V_{2},\delta_{2})$. By Theorem \ref{Quantifier Elimination of T Convex Delta G} and saturation, there exists $b \in M_{2}$ such that
\[
\tp^{(\mathcal{M}_{1},V_{1},\delta_{1})}(a/M)
=
\tp^{(\mathcal{M}_{2},V_{2},\delta_{2})}(b/M).
\]
Hence there is an $\mathcal{L}_{\mathrm{convex}}^{\delta}$-isomorphism
\[
h:(\mathcal{M}\langle\Jet_{\infty}(a)\rangle,V_{a},\delta_{1})
\longrightarrow
(\mathcal{M}\langle\Jet_{\infty}(b)\rangle,V_{b},\delta_{2})
\]
over $(\mathcal{M},V,\delta)$, where $V_{a}$ and $V_{b}$ are the convex hulls induced on the corresponding generated substructures.

Let $\mathcal{H}$ be the $\mathcal{L}_{\mathrm{tame}}^{\delta}$-substructure of $(\mathcal{M}_{1},\mathcal{N}_{1},\st_{1},\delta_{1})$ generated by $M \cup \{a\}$, and define $\mathcal{H}'$ similarly using $M \cup \{b\}$. By Lemma 5.8 and the proof of Theorem 5.9 in \cite{vandendriesLewenberg1995}, the $\mathcal{L}_{\mathrm{convex}}$-isomorphism underlying $h$ induces an $\mathcal{L}_{\mathrm{tame}}$-isomorphism from $\mathcal{H}$ to $\mathcal{H}'$ over $(\mathcal{M},\mathcal{N},\st)$. If $x \in V_{a}$, then $\st_{1}(x)$ is the unique element of the distinguished substructure such that $x-\st_{1}(x)$ belongs to the maximal ideal of $V_{a}$. Since $h$ preserves the convex hull and its maximal ideal, we have
\[
h(\st_{1}(x))=\st_{2}(h(x)).
\]
If $x \notin V_{a}$, then $h(x) \notin V_{b}$ and both standard part maps take value $0$. Thus $h$ induces an $\mathcal{L}_{\mathrm{tame}}^{\delta}$-isomorphism from $\mathcal{H}$ to $\mathcal{H}'$ over $(\mathcal{M},\mathcal{N},\st,\delta)$. Since $a \notin M$, the structure $\mathcal{H}$ properly extends $(\mathcal{M},\mathcal{N},\st,\delta)$.
\end{proof}

\begin{theorem} \label{T tame G QE}
Suppose that $T$ has quantifier elimination and is universally axiomatizable. Then, the theory $T_{g,\mathrm{tame}}^{\delta}$ has quantifier elimination.
\end{theorem}

\begin{proof}
We employ the test for quantifier elimination (Fact \ref{Quantifier Elimination Test}). Fix $\mathcal{L}_{\mathrm{tame}}^{\delta}$-structures $(\mathcal{M}_{1},\mathcal{N}_{1},\st_{1},\delta_{1}),(\mathcal{M}_{2},\mathcal{N}_{2},\st_{2},\delta_{2}) \models T_{g,\mathrm{tame}}^{\delta}$, and let $(\mathcal{M},\mathcal{N},\st,\delta)$ be a common substructure of both with $M \neq M_{1}$. We may assume that $(\mathcal{M}_{2},\mathcal{N}_{2},\st_{2},\delta_{2})$ is sufficiently saturated. By Corollary \ref{Universal Part of T tame delta G}, we have $(\mathcal{M},\mathcal{N},\st,\delta) \models T_{\mathrm{tame}}^{\delta,-}$. By Lemma \ref{Embedding Lemma for T tame delta G}, there is an $\mathcal{L}_{\mathrm{tame}}^{\delta}$-substructure of $(\mathcal{M}_{1},\mathcal{N}_{1},\st_{1},\delta_{1})$ properly extending $(\mathcal{M},\mathcal{N},\st,\delta)$ which embeds into $(\mathcal{M}_{2},\mathcal{N}_{2},\st_{2},\delta_{2})$ over $(\mathcal{M},\mathcal{N},\st,\delta)$. Therefore, $T_{g,\mathrm{tame}}^{\delta}$ has quantifier elimination.
\end{proof}

\begin{corollary} \label{T tame G Complete and Model Complete}
The theory $T_{g,\text{tame}}^{\delta}$ is complete and model complete (recall that by assumption, $T$ is complete and model complete).
\end{corollary}

\begin{proof}
The same arguments as in Corollary \ref{Model Completeness of T Convex Delta G} apply.
\end{proof}

\subsection{Stable Embedding}

We first record a preparation lemma for the stable embedding property.

\begin{lemma} \label{Stable Embedding of T Tame delta G Preparation}
Let $(\mathcal{M},\mathcal{N},\st,\delta) \preceq (\mathcal{M}_{1},\mathcal{N}_{1},\st_{1},\delta_{1}),(\mathcal{M}_{2},\mathcal{N}_{2},\st_{2},\delta_{2}) \models T_{g,\mathrm{tame}}^{\delta}$. Let $n>0$, $\overline{a} \in N_{1}^{n}$, and $\overline{b} \in N_{2}^{n}$. Suppose that $\tp^{(\mathcal{M}_{1},\delta_{1})}(\overline{a}/N) = \tp^{(\mathcal{M}_{2},\delta_{2})}(\overline{b}/N)$. Then
\[
\tp^{(\mathcal{M}_{1},\mathcal{N}_{1},\st_{1},\delta_{1})}(\overline{a}/M)
=
\tp^{(\mathcal{M}_{2},\mathcal{N}_{2},\st_{2},\delta_{2})}(\overline{b}/M).
\]
\end{lemma}

\begin{proof}
For every $m \in \mathbb{N}$, the assumption gives
\[
\tp^{\mathcal{M}_{1}}(\Jet_{m}(\overline{a})/N)
=
\tp^{\mathcal{M}_{2}}(\Jet_{m}(\overline{b})/N).
\]
Since $\overline{a} \in N_{1}^{n}$ and $\overline{b} \in N_{2}^{n}$, and since $N_{1}$ and $N_{2}$ are closed under the corresponding derivations, we have $\Jet_{m}(\overline{a}) \in N_{1}^{n(m+1)}$ and $\Jet_{m}(\overline{b}) \in N_{2}^{n(m+1)}$. By Lemma \ref{Stable Embedding of T Tame Preparation},
\[
\tp^{(\mathcal{M}_{1},\mathcal{N}_{1},\st_{1})}(\Jet_{m}(\overline{a})/M)
=
\tp^{(\mathcal{M}_{2},\mathcal{N}_{2},\st_{2})}(\Jet_{m}(\overline{b})/M)
\]
for every $m \in \mathbb{N}$.

Let $\varphi(\overline{x})$ be an $\mathcal{L}_{\mathrm{tame}}^{\delta}(M)$-formula. By Theorem \ref{T tame G QE}, we may assume that $\varphi$ is quantifier-free. Since $\varphi$ contains only finitely many occurrences of $\delta$, there exist $m \in \mathbb{N}$ and an $\mathcal{L}_{\mathrm{tame}}(M)$-formula $\widetilde{\varphi}$ such that
\[
\varphi(\overline{x})
\quad\text{is equivalent to}\quad
\widetilde{\varphi}(\Jet_{m}(\overline{x})).
\]
It follows from the equality of the $\mathcal{L}_{\mathrm{tame}}$-types of the finite jets that $(\mathcal{M}_{1},\mathcal{N}_{1},\st_{1},\delta_{1}) \models \varphi(\overline{a})$ if and only if $(\mathcal{M}_{2},\mathcal{N}_{2},\st_{2},\delta_{2}) \models \varphi(\overline{b})$. Therefore,
\[
\tp^{(\mathcal{M}_{1},\mathcal{N}_{1},\st_{1},\delta_{1})}(\overline{a}/M)
=
\tp^{(\mathcal{M}_{2},\mathcal{N}_{2},\st_{2},\delta_{2})}(\overline{b}/M).
\]
\end{proof}

The next proposition establishes that the theory $T_{g,\text{tame}}^{\delta}$ has the stable embedding property.

\begin{proposition} \label{Stable Embedding of T Tame delta G}
Let $(\mathcal{M},\mathcal{N},\st,\delta) \models T_{g,\text{tame}}^{\delta}$ and $n \in \mathbb{N}$. If $X \subseteq M^{n}$ is definable in $(\mathcal{M},\mathcal{N},\st,\delta)$, then $X \cap N^{n}$ is definable in $(\mathcal{N},\delta|_{N})$.
\end{proposition}

\begin{proof}
The proposition follows from Fact \ref{Stone Duality Theorem} and Lemma \ref{Stable Embedding of T Tame delta G Preparation}.
\end{proof}

\subsection{NIP of $T^{\delta}_{g,\text{tame}}$}

In this section, we show that the theory $T^{\delta}_{g,\text{tame}}$ has NIP. Let $\mathcal{L}^{\delta}_{U}$ be the language $\mathcal{L}^{\delta}$ expanded by a unary predicate $U$, and let $T^{\delta}_{g,\text{tame},-}$ be the reduct of $T^{\delta}_{g,\text{tame}}$ to $\mathcal{L}^{\delta}_{U}$. The standard part map is definable in this reduct. Indeed, ``$\st(x)=y$'' is equivalent to
\[
Uy \wedge
\left[
\forall z\,[(Uz \wedge 0<z)\rightarrow |x-y|<z]
\vee
\left(y=0\wedge\forall z\,[Uz\rightarrow |z|<|x|]\right)
\right].
\]
The first disjunct is the case where $x$ is $N$-bounded, while the second is the case where $x$ is infinite with respect to $N$, in which case $\st(x)=0$. Thus, there is a bijective correspondence between models of $T^{\delta}_{g,\text{tame}}$ and models of $T^{\delta}_{g,\text{tame},-}$. We use a result on $T^{\delta}_{g}$ by Fornasiero and Kaplan in \cite{FornasieroKaplan2020}, together with the NIP test developed by Chernikov and Simon in \cite{ChernikovSimon2012}.

\begin{definition}
An $\mathcal{L}^{\delta}_{U}$-formula is \textbf{$\mathcal{L}^{\delta}$-bounded} if it is of the form
\[
Q_{1}y_{1}\in U\cdots Q_{n}y_{n}\in U\,\varphi(\overline{x},y_{1},\ldots,y_{n}),
\]
where each $Q_{i}\in\{\forall,\exists\}$ and $\varphi(\overline{x},\overline{y})$ is an $\mathcal{L}^{\delta}$-formula. The theory $T^{\delta}_{g,\text{tame},-}$ is \textbf{$\mathcal{L}^{\delta}$-bounded} if every $\mathcal{L}^{\delta}_{U}$-formula is $T^{\delta}_{g,\text{tame},-}$-equivalent to an $\mathcal{L}^{\delta}$-bounded formula.
\end{definition}

\begin{fact}[Corollary 2.6 in \cite{ChernikovSimon2012}]
Assume that $T^{\delta}_{g}$ has NIP. If $T^{\delta}_{g,\text{tame},-}$ is $\mathcal{L}^{\delta}$-bounded, then $T^{\delta}_{g,\text{tame},-}$ has NIP.
\end{fact}

\begin{fact}[Corollary 4.16 in \cite{FornasieroKaplan2020}]
$T^{\delta}_{g}$ has NIP.
\end{fact}

Thus, it suffices to show that $T^{\delta}_{g,\text{tame},-}$ is $\mathcal{L}^{\delta}$-bounded. We may assume that $\mathcal{L}=\mathcal{L}^{\mathrm{df}}$ and $T=T^{\mathrm{df}}$, so that $T^{\delta}_{g,\text{tame}}$ has quantifier elimination.

We first observe that the graph of the standard part map is $\mathcal{L}^{\delta}$-bounded. For $y\in U$, the formula $\st(x)=y$ is equivalent to
\[
\left[\forall z\in U\,(0<z\rightarrow |x-y|<z)\right]
\vee
\left[y=0\wedge\forall z\in U\,(|z|<|x|)\right].
\]

We now show that every atomic $\mathcal{L}^{\delta}_{\text{tame}}$-formula is equivalent to an $\mathcal{L}^{\delta}$-bounded formula. We proceed by induction on the number of occurrences of $\st$. If there is no occurrence of $\st$, then the formula is an atomic $\mathcal{L}^{\delta}_{U}$-formula. A formula of the form $Ut(\overline{x})$ is equivalent to
\[
\exists y\in U\,[y=t(\overline{x})],
\]
and every other atomic formula is already an $\mathcal{L}^{\delta}$-formula.

Suppose now that $\st$ occurs in the formula. Choose an innermost occurrence $\st(t(\overline{x}))$, where $t(\overline{x})$ is an $\mathcal{L}^{\delta}$-term, and write the atomic formula as $\psi(\overline{x},\st(t(\overline{x})))$, where $\psi(\overline{x},y)$ has one fewer occurrence of $\st$. This formula is equivalent to
\[
\exists y\in U\,
\left[
\left(
\forall z\in U\,[0<z\rightarrow |t(\overline{x})-y|<z]
\vee
\left[y=0\wedge\forall z\in U\,(|z|<|t(\overline{x})|)\right]
\right)
\wedge\psi(\overline{x},y)
\right].
\]
By the induction hypothesis, this is equivalent to an $\mathcal{L}^{\delta}$-bounded formula. Hence every atomic $\mathcal{L}^{\delta}_{\text{tame}}$-formula is $\mathcal{L}^{\delta}$-bounded.

The class of $\mathcal{L}^{\delta}$-bounded formulas is closed under Boolean combinations, after moving bounded quantifiers outward and replacing negated bounded quantifiers by their duals. Therefore, every quantifier-free $\mathcal{L}^{\delta}_{\text{tame}}$-formula is $\mathcal{L}^{\delta}$-bounded. Since $T^{\delta}_{g,\text{tame}}$ has quantifier elimination, it follows that $T^{\delta}_{g,\text{tame},-}$ is $\mathcal{L}^{\delta}$-bounded.

\begin{proposition} \label{T delta g tame has NIP}
$T^{\delta}_{g,\text{tame}}$ has NIP.
\end{proposition}

\begin{proof}
By the two facts above, $T^{\delta}_{g,\text{tame},-}$ has NIP. Since the standard part map is definable in this reduct, $T^{\delta}_{g,\text{tame}}$ is a definitional expansion of $T^{\delta}_{g,\text{tame},-}$. Therefore, $T^{\delta}_{g,\text{tame}}$ has NIP.
\end{proof}

%% file: Definability_of_Hausdorff_Limits.tex
\section{Metric Topologies on $T^{\delta}_{g}$ and Definability of Hausdorff Limits} \label{Section 5}

In this section, we characterize definable types over subsets of models of $T^{\delta}_{g}$ by adapting the Marker-Steinhorn Theorem. We also provide a geometric interpretation of the Marker-Steinhorn Theorem and the stable embedding property of $T^{\delta}_{g}$ by investigating the definability of Hausdorff limits. Classical results on the definability of Hausdorff limits typically concern the Euclidean distance function in some o-minimal expansion of the real field. We could transfer these results directly, but since $T^{\delta}_{g}$ has $T$ as its open core (see Section 5.3 in \cite{FornasieroKaplan2020}), every Euclidean open $\mathcal{L}^{\delta}$-definable set is already $\mathcal{L}$-definable. This makes the definability results with respect to the Euclidean metric somewhat less interesting. To address this, we introduce the sequence of $\mathcal{L}^{\delta}(\varnothing)$-definable distance functions that are more natural in the context of differential fields.

\subsection{Marker-Steinhorn Theorem}

Given a language $\mathcal{L}_{1}$, $n \in \mathbb{N}$, an $\mathcal{L}_{1}$-structure $\mathcal{M}$, and a subset $A \subseteq M$, a type $p(\overline{x})$ in the type space $S_{n}^{\mathcal{M}}(A)$ is said to be \textbf{$A$-definable} if for any $\mathcal{L}_{1}(\varnothing)$-formula $\varphi(\overline{x},\overline{y})$, there exists an $\mathcal{L}_{1}(A)$-formula $d\varphi(\overline{y})$ such that for all $\overline{b} \in A$, we have $\varphi(\overline{x},\overline{b}) \in p(\overline{x})$ if and only if $\mathcal{M} \models d\varphi(\overline{b})$. Definability of types is a central topic in model theory. For instance, one of the key features of stable theories is that a theory $T_{1}$ is stable if and only if for any $n \in \mathbb{N}$ and any $\mathcal{M} \models T_{1}$, every $p \in S_{n}^{\mathcal{M}}(M)$ is definable. In non-stable theories, however, understanding what makes a type definable is more subtle. 

A characterization of definable types in o-minimal structures was first given by Marker and Steinhorn \cite{MarkerSteinhorn1994}, hence the name Marker-Steinhorn Theorem. Later, Pillay \cite{Pillay1994}, Tressl \cite{Tressl2004}, van den Dries \cite{Lisbon2003}, and Walsberg \cite{Walsberg2019} provided variants of the proof for either the general or special cases. It is worth noting that a gap in the proof of the general case existed for some time, which was recently fixed by Guerrero \cite{AndújarGuerrero2025} with a simpler proof.

In this section, we provide a characterization of definable types in models of $T^{\delta}_{g}$. The special case for $\CODF$ was studied by Brouette in \cite{Brouette2016}. We show that his proof easily transfers to the general case and explain how the Marker-Steinhorn Theorem for $T^{\delta}_{g}$ relates to the stable embedding property of $T^{\delta}_{g}$. Fix a model $(\mathcal{M},\delta) \models T^{\delta}_{g}$ and a subset $A \subseteq M$.

\begin{definition}
A type $p(x) \in S^{\mathcal{M}}_{1}(A)$ is a \textbf{nonprincipal cut of $A$} if there exist nonempty disjoint subsets $C_{0}$ and $C_{1}$ of $A$ such that $C_{0} \cup C_{1} = A$, $C_{0}$ has no greatest element, $C_{1}$ has no least element, for all $c \in C_{0}$ we have $``c < x" \in p(x)$, and for all $c \in C_{1}$ we have $``x < c" \in p(x)$. If $p(x)$ is nonisolated and not a nonprincipal cut, then it is called a \textbf{principal cut of $A$}.
\end{definition}

\begin{definition}
We say that $A$ is \textbf{Dedekind complete in $M$} if for every $a \in M$, the type $\tp^{\mathcal{M}}(a/A)$ is not a nonprincipal cut. We also say that $A$ is \textbf{Dedekind complete} if $A$ is Dedekind complete in every elementary extension of $\mathcal{M}$.
\end{definition}

The following is the original Marker-Steinhorn Theorem for o-minimal structures.

\begin{theorem}[Marker-Steinhorn, Theorem 2.1 in \cite{MarkerSteinhorn1994}] \label{Marker-Steinhorn}
Let $p(\overline{x}) \in S^{\mathcal{M}}_{n}(M)$. Then $p(\overline{x})$ is $M$-definable if and only if $M$ is Dedekind complete in $M(\overline{a})$, where $\overline{a}$ is any $n$-tuple realizing $p(\overline{x})$ and $M(\overline{a})$ is the prime model of $\Th(\mathcal{M})$ generated by $M$ and $\overline{a}$.
\end{theorem}

The Marker–Steinhorn Theorem characterizes the definability of types in an o-minimal structure in terms of whether the type realizes a nonprincipal cut. Our goal is to adapt these ideas to models of $T^{\delta}_{g}$. We begin by modifying the relevant definitions.

\begin{definition}
A type $p(x) \in S^{(\mathcal{M},\delta)}_{1}(A)$ is called a \textbf{nonprincipal differential cut of $A$} if there exist an element $a$ in some elementary extension of $(\mathcal{M},\delta)$ and $k \in \mathbb{N}$ such that $a$ realizes $p$ and $\tp^{\mathcal{M}}(\delta^{k} a/A)$ is a nonprincipal cut. If $p(x)$ is nonisolated and not nonprincipal, then it is called a \textbf{principal differential cut of $A$}.
\end{definition}

Before adapting the Marker–Steinhorn Theorem to our setting, we require an auxiliary lemma concerning $T^{\delta}_{g}$, which follows easily from the fact that $T^{\delta}_{g}$ has $T$ as its open core.

\begin{fact}[Corollary 5.13 in \cite{FornasieroKaplan2020}] \label{Definable Complete}
Every model of $T^{\delta}_{g}$ is definably complete. In particular, if $D \subseteq M$ is $\mathcal{L}^{\delta}(M)$-definable, then the supremum of $D$ belongs to $M \cup \{+\infty\}$.
\end{fact}

\begin{corollary} \label{Definable Complete and Definable Closure}
If $D \subseteq A$ is $\mathcal{L}^{\delta}(M)$-definable and $A = \dcl_{(\mathcal{M},\delta)}(A)$, then the supremum of $D$ belongs to $A \cup \{+\infty\}$.
\end{corollary}

Explicitly, the set $\dcl_{(\mathcal{M},\delta)}(A)$ is precisely $\dcl_{\mathcal{M}}(\Jet_{\infty}(A))$, which follows easily from Lemma \ref{delta on Ck-functions}. We now adapt Brouette's proofs for $\CODF$ to our setting. The arguments are very similar, but for the sake of completeness, we present the general version in full.

\begin{lemma} \label{Definable Types are not Cuts}
Let $A=\dcl_{(\mathcal{M},\delta)}(A)$. If $p(x)\in S_{1}^{(\mathcal{M},\delta)}(A)$ is $A$-definable, then $p(x)$ is not a nonprincipal differential cut of $A$. In particular, if $p(x)$ is nonisolated, then it is a principal differential cut of $A$.
\end{lemma}

\begin{proof}
Fix $n \in \mathbb{N}$ and let $\varphi(x,b)$ denote the formula $``b < \delta^{n}x"$.  
By definability of $p(x)$, there exists an $\mathcal{L}^{\delta}(A)$-formula $d\varphi$ such that
\[
\{ b \in A \mid \varphi(x,b) \in p(x) \} = d\varphi(A),
\]
where $d\varphi(A) := \{a \in A \mid (\mathcal{M},\delta) \models d\varphi(a) \}$. By Corollary~\ref{Definable Complete and Definable Closure}, if $d\varphi(A)$ is nonempty, then its supremum lies in $A \cup \{+\infty\}$. Hence, if $c$ realizes $p(x)$, the type $\tp^{\mathcal{M}}(\delta^{n}c / A)$ is not a nonprincipal cut.  
Since $n$ was arbitrary, it follows that $p(x)$ is not a nonprincipal differential cut.
\end{proof}

\begin{lemma} \label{Definable Types implies Dedekind Completeness}
Suppose that $A = \dcl_{(\mathcal{M},\delta)}(A)$ and $\overline{a} \in M$ are such that $\tp^{(\mathcal{M},\delta)}(\overline{a}/A)$ is $A$-definable. Then $A$ is Dedekind complete in $\dcl_{(\mathcal{M},\delta)}(A,\overline{a})$.
\end{lemma}

\begin{proof}
Suppose otherwise. Then there exists $b \in \dcl_{(\mathcal{M},\delta)}(A,\overline{a})$ such that $C_{0} < b < C_{1}$, where $C_{0}$ and $C_{1}$ are nonempty disjoint subsets of $A$ with $C_{0} \cup C_{1} = A$, $C_{0}$ has no greatest element, $C_{1}$ has no least element, for all $c \in C_{0}$ we have $``c < x" \in p(x)$, and for all $c \in C_{1}$ we have $``x < c" \in p(x)$. This implies that there exists an $\mathcal{L}^{\delta}(\varnothing)$-definable function $f$ such that $b = f(\overline{c},\overline{a})$ for some $\overline{c} \in A$. Let $p(x) := \tp^{(\mathcal{M},\delta)}(b/A)$. Then $p(x)$ is a nonprincipal differential cut. Note that for any $\mathcal{L}^{\delta}(\varnothing)$-formula $\varphi(x,\overline{y})$ and $\overline{d} \in A$, we have $\varphi(x,\overline{d}) \in p(x)$ if and only if $\varphi(f(\overline{c},\overline{y}),\overline{d}) \in \tp^{(\mathcal{M},\delta)}(\overline{a}/A)$. Since $\tp^{(\mathcal{M},\delta)}(\overline{a}/A)$ is $A$-definable, this implies that $p(x)$ is $A$-definable, contradicting Lemma~\ref{Definable Types are not Cuts}.
\end{proof}

\begin{lemma} \label{Dedekind Completeness implies Definable Types}
Suppose that $A = \dcl_{(\mathcal{M},\delta)}(A)$ and $\overline{a} \in M$ are such that $A$ is Dedekind complete in $\dcl_{(\mathcal{M},\delta)}(A,\overline{a})$. Then $\tp^{(\mathcal{M},\delta)}(\overline{a}/A)$ is $A$-definable.
\end{lemma}

\begin{proof}
By assumption, for any $n \in \mathbb{N}$, $A$ is Dedekind complete in $\dcl_{\mathcal{M}}(A,\Jet_{n}(\overline{a}))$. Note that 
\[
A \subseteq \dcl_{\mathcal{M}}(A) \subseteq \dcl_{(\mathcal{M},\delta)}(A) = A,
\]
and since $\mathcal{M}$ is an o-minimal structure expanding a real closed field, $\dcl_{\mathcal{M}}(A)$ is an elementary substructure of $\mathcal{M}$. Thus, by Theorem~\ref{Marker-Steinhorn}, the type $\tp^{\mathcal{M}}(\Jet_{n}(\overline{a})/A)$ is $A$-definable in the language $\mathcal{L}$. Fix an $\mathcal{L}^{\delta}(\varnothing)$-formula $\varphi(\overline{x},\overline{y})$ and $\overline{b} \in A$. By Lemma~\ref{Forget the Derivation in T delta G}, there exist $k,l \in \mathbb{N}$ and an $\mathcal{L}(\varnothing)$-formula $\psi$ such that
\[
(\mathcal{M},\delta) \models \forall \overline{x} [\varphi(\overline{x},\overline{b}) \leftrightarrow \psi(\Jet_{k}(\overline{x}),\Jet_{l}(\overline{b}))],
\]
and hence
\[
\varphi(\overline{x},\overline{b}) \in \tp^{(\mathcal{M},\delta)}(\overline{a}/A) \Leftrightarrow \psi(\Jet_{k}(\overline{x}),\Jet_{l}(\overline{b})) \in \tp^{\mathcal{M}}(\Jet_{k}(\overline{a})/A)
\]
\[
\Leftrightarrow \mathcal{M} \models d\psi(\Jet_{l}(\overline{b})) \Leftrightarrow (\mathcal{M},\delta) \models d\psi(\Jet_{l}(\overline{b})).
\]
\end{proof}

We have now proved both directions of the Marker-Steinhorn Theorem for models of $T^{\delta}_{g}$. The following proposition restates both lemmas in a unified form.

\begin{proposition} \label{Marker-Steinhorn for Derivation}
Suppose that $A = \dcl_{(\mathcal{M},\delta)}(A)$ and $\overline{a} \in M$. Then $\tp^{(\mathcal{M},\delta)}(\overline{a}/A)$ is $A$-definable if and only if $A$ is Dedekind complete in $\dcl_{(\mathcal{M},\delta)}(A,\overline{a})$.
\end{proposition}

\begin{remark}
One might think that Proposition~\ref{Marker-Steinhorn for Derivation} is an immediate consequence of Proposition~\ref{Stable Embedding of T Tame delta G}. This is indeed true when $A$ is the universe of an elementary substructure of $(\mathcal{M},\delta)$. However, Proposition~\ref{Marker-Steinhorn for Derivation} also applies to the more general case where $A$ is merely an $\mathcal{L}^{\delta}$-definably closed subset of $M$.
\end{remark}

If $A = \dcl_{(\mathcal{M},\delta)}(A)$, then by o-minimality and the field structure, the $\mathcal{L}$-structure $\mathcal{A}$ with universe $A$ and structure inherited from $\mathcal{M}$ is a model of $T$. Hence, $(\mathcal{A},\delta|_{A}) \models T^{\delta}$ is an $\mathcal{L}^{\delta}$-substructure of $(\mathcal{M},\delta)$. This allows us to restate Proposition~\ref{Marker-Steinhorn for Derivation} in a geometric form as follows.

\begin{corollary} \label{Marker-Steinhorn for Derivation Geometric}
Suppose that $A = \dcl_{(\mathcal{M},\delta)}(A)$ and $\overline{a} \in M$. Then for every $\mathcal{L}^{\delta}(A\cup\{\overline{a}\})$-definable set $X \subseteq M^{n}$ in $(\mathcal{M},\delta)$, its trace $X \cap A^{n}$ is definable in $(\mathcal{A},\delta|_{A})$ if and only if $A$ is Dedekind complete in $\dcl_{(\mathcal{M},\delta)}(A,\overline{a})$.
\end{corollary}

\begin{proof}
``$\Leftarrow$'': Suppose that $A$ is Dedekind complete in $\dcl_{(\mathcal{M},\delta)}(A,\overline{a})$. By Proposition \ref{Marker-Steinhorn for Derivation}, the type $\tp^{(\mathcal{M},\delta)}(\overline{a}/A)$ is $A$-definable. Let $X \subseteq M^{n}$ be defined by an $\mathcal{L}^{\delta}(A)$-formula $\varphi(\overline{a},\overline{y})$. Then there exists an $\mathcal{L}^{\delta}(A)$-formula $d\varphi(\overline{y})$ such that for every $\overline{b} \in A^{n}$,
\[
\overline{b} \in X \cap A^{n}
\quad\Longleftrightarrow\quad
(\mathcal{M},\delta) \models d\varphi(\overline{b}).
\]
After replacing $\mathcal{L}$ by $\mathcal{L}^{\mathrm{df}}$ and $T$ by $T^{\mathrm{df}}$, we may assume that $T^{\delta}_{g}$ has quantifier elimination. Hence, we may take $d\varphi$ to be quantifier-free. Since $(\mathcal{A},\delta|_{A})$ is an $\mathcal{L}^{\delta}$-substructure of $(\mathcal{M},\delta)$, every quantifier-free formula with parameters from $A$ has the same truth value in $(\mathcal{A},\delta|_{A})$ and $(\mathcal{M},\delta)$. Therefore, $d\varphi(\overline{y})$ defines $X \cap A^{n}$ in $(\mathcal{A},\delta|_{A})$.

``$\Rightarrow$'': Suppose that for every $\mathcal{L}^{\delta}(A\cup\{\overline{a}\})$-definable set $X \subseteq M^{n}$ in $(\mathcal{M},\delta)$, its trace $X \cap A^{n}$ is definable in $(\mathcal{A},\delta|_{A})$. Fix an $\mathcal{L}^{\delta}(A)$-formula $\varphi(\overline{x},\overline{y})$. Let $X$ be the set of all $\overline{b} \in M^{n}$ such that $\varphi(\overline{x},\overline{b}) \in \tp^{(\mathcal{M},\delta)}(\overline{a}/M)$. Then $X$ is defined by the formula $\varphi(\overline{a},\overline{y})$ in $(\mathcal{M},\delta)$. By assumption, its trace $X \cap A^{n}$ is defined by an $\mathcal{L}^{\delta}(A)$-formula $d\varphi(\overline{y})$. This shows that the type $\tp^{(\mathcal{M},\delta)}(\overline{a}/A)$ is $A$-definable. Thus, $A$ is Dedekind complete in $\dcl_{(\mathcal{M},\delta)}(A,\overline{a})$ by Proposition~\ref{Marker-Steinhorn for Derivation}.
\end{proof}

\subsection{Metric Topologies and Hausdorff Distance Functions}

In this section, we introduce a sequence of metrics that are more natural for models of $T^{\delta}_{g}$. The Euclidean topology is the natural topology associated with o-minimal structures for many reasons (for details, see \cite{vandenDries2003}). For models of $T^{\delta}_{g}$, the Euclidean topology appears to be too coarse since $\delta$ is highly discontinuous with respect to it. In \cite{BrihayeMichauxRivière2009}, Brihaye, Michaux, and Rivière introduced the $\delta$-topology on models of $\CODF$, which is the coarsest topology containing the Euclidean topology that makes $\delta$ a continuous function. The $\delta$-topology can be easily generalized to models of $T^{\delta}_{g}$ (see \cite{BrihayeMichauxRivière2009} for a more detailed description of the $\delta$-topology on $\CODF$). However, the $\delta$-topology is not a definable topology in the language $\mathcal{L}^{\delta}$. To see this, suppose otherwise. Fix $n \in \mathbb{N}$. Then, there exists an $\mathcal{L}^{\delta}(\varnothing)$-formula $\varphi(\overline{x},\overline{y})$ such that the definable family 
\[
\mathcal{O} := \{C \subseteq M^{n}|\text{$C$ is defined by the formula }\varphi(\overline{a},\overline{y}), \text{for some }\overline{a} \in M\}
\]
forms a neighbourhood basis at the origin in $M^{n}$. By replacing $T$ with $T^{\text{df}}$, we may assume that $T^{\delta}_{g}$ admits quantifier elimination, and hence $\varphi(\overline{x},\overline{y})$ is quantifier-free. Let $k$ be the largest nonnegative integer such that $\delta^{k}$ appeared in $\varphi$. Since $\mathcal{O}$ is a neighbourhood basis at $\overline{0}$, the open set (with respect to the $\delta$-topology) 
\[
\{\overline{y} \in M^{n}|-1<\delta^{i}y_{j}<1,0 \leq i \leq k+1, 1 \leq j \leq n\}
\]
contains some $C \in \mathcal{O}$. This is impossible since by the axiom of genericity, $C$ must contains an element $\overline{y}$ such that $\delta^{k+1}y_{1} \geq 1$. This makes the treatment of the $\delta$-topology significantly different from that of the Euclidean topology in o-minimal structures. Here, we introduce a sequence of metric topologies that are $\mathcal{L}^{\delta}(\varnothing)$-definable and approximate the $\delta$-topology.

For the rest of this section, fix $(\mathcal{M},\delta) \models T^{\delta}_{g}$.

\begin{definition}
For each $n,k \in \mathbb{N}$, let $d_{n,k}:M^{k} \times M^{k} \rightarrow M$ be the map
\[
d_{n,k}(\overline{x},\overline{y}) = \sum_{i=0}^{n}\sqrt{\sum_{j=1}^{k}(\delta^{i}x_{j}-\delta^{i}y_{j})^{2}},
\]
and let $\widetilde{d}_{n,k}:M^{k} \times M^{k} \rightarrow M$ be the map
\[
\widetilde{d}_{n,k}(\overline{x},\overline{y}) = \sum_{i=0}^{n}2^{-i}\frac{\sqrt{\sum_{j=1}^{k}(\delta^{i}x_{j}-\delta^{i}y_{j})^{2}}}{1+\sqrt{\sum_{j=1}^{k}(\delta^{i}x_{j}-\delta^{i}y_{j})^{2}}}.
\]
We call $d_{n,k}$ the \textbf{\textit{standard $(n,k)$-metric}} and $\widetilde{d}_{n,k}$ the \textbf{\textit{rescaled standard $(n,k)$-metric}} on $M^{k}$. When $k$ is clear from the context, we omit it.
\end{definition}

It is not hard to see that $d_{n,k}$ and $\widetilde{d}_{n,k}$ are indeed $M$-valued distance functions on $M^{k}$, and that for each $n,k$, the metrics $d_{n,k}$ and $\widetilde{d}_{n,k}$ induce the same $M$-valued metric topology on $M^{k}$. Let $\mathcal{T}_{n,k}$ denote the topology induced by $d_{n,k}$. Then it is clear that $\mathcal{T}_{n,k} \subseteq \mathcal{T}_{n+1,k}$ for all $n \in \mathbb{N}$, and the $\delta$-topology on $M^{k}$, denoted by $\mathcal{T}_{\infty,k}$, is the topology generated by $\bigcup_{n\in\mathbb{N}}\mathcal{T}_{n,k}$. The following is an immediate consequence of this discussion.

\begin{proposition}
Let $X \subseteq M^{k}$ be $\mathcal{L}^{\delta}(M)$-definable. Let $\cl_{n}(X)$ denote the closure of $X$ in the standard $(n,k)$-metric topology, and let $\cl_{\delta}(X)$ denote the closure of $X$ in the $\delta$-topology. Then for each $n \in \mathbb{N}$, we have $\cl_{n+1}(X) \subseteq \cl_{n}(X)$, and
\[
\cl_{\delta}(X) = \bigcap_{n=0}^{\infty}\cl_{n}(X).
\]
\end{proposition}

\begin{remark}
In the case where $M = \mathbb{R}$, we can define
\[
\widetilde{d}_{\infty,k}(\overline{x},\overline{y}) := \sum_{i=0}^{\infty}2^{-i}\frac{\sqrt{\sum_{j=1}^{k}(\delta^{i}x_{j}-\delta^{i}y_{j})^{2}}}{1+\sqrt{\sum_{j=1}^{k}(\delta^{i}x_{j}-\delta^{i}y_{j})^{2}}}.
\]
The series converges and is therefore well-defined. It is not hard to see that this metric induces the $\delta$-topology on $\mathbb{R}^{k}$.
\end{remark}

Fix an $\mathcal{L}^{\delta}(M)$-definable $M$-metric space $(Z,d)$. Let $\mathcal{K}(Z)$ denote the set of all nonempty $\mathcal{L}^{\delta}(M)$-definable closed and bounded subsets of $Z$. The \textbf{Hausdorff distance function $D:\mathcal{K}(Z) \times \mathcal{K}(Z) \rightarrow M$ induced by $d$} is defined by 
\[
D(A,B) = \max\{\inf \{\varepsilon \in M \mid B \subseteq U_{d}(A,\varepsilon)\},\inf \{\varepsilon \in M \mid A \subseteq U_{d}(B,\varepsilon)\}\}.
\]
Intuitively, the Hausdorff distance function $D$ measures how far two sets are from each other with respect to the distance function $d$. It is straightforward to verify that

\begin{fact} \label{Charaterization of Hausdorff Metric}
$D(A,B) = \max\left\{\sup_{a \in A}\inf_{b \in B}d(a,b),\sup_{b \in B}\inf_{a \in A}d(a,b)\right\}$. In particular, the value $D(A,B)$ is $\mathcal{L}^{\delta}(M)$-definable for every pair of sets $A,B \in \mathcal{K}(Z)$.
\end{fact}

\subsection{Definability of Hausdorff Limits}

In this section, we give the geometric interpretation of the Marker-Steinhorn Theorem and the stable embedding property of $T^{\delta}_{g}$. Van den Dries has established these results for o-minimal structures in \cite{Lisbon2003} and \cite{vandenDries1997}, and we will adapt his methods to models of $T^{\delta}_{g}$. Let $(\mathcal{M},\delta) \models T_{g}^{\delta}$ be $\aleph_{1}$-saturated such that $\mathbb{R} \subseteq M$ and $\dcl_{(\mathcal{M},\delta)}(\mathbb{R}) = \mathbb{R}$. In particular, $(\mathbb{R},\delta|_{\mathbb{R}}) \models T^{\delta}$. Fix $n \in \mathbb{N}$. Let $A^{*} \subseteq \mathcal{M}^{m+k}$ be $\mathcal{L}^{\delta}(\varnothing)$-definable. Recall that $\Pi_{m}: \mathcal{M}^{m+k} \rightarrow \mathcal{M}^{m}$ denotes the projection onto the first $m$ coordinates. Set $(A')^{*} := \Pi_{m}(A^{*})$. For each $a \in (A')^{*}$, denote
\[
A^{*}_{a} := \{x \mid (a,x) \in A^{*}\}, \quad F(A^{*}) := \{A^{*}_{a} \mid a \in (A')^{*}\}.
\]
Since $\mathbb{R}$ is Dedekind complete, we can define $\st_{n,k}:M^{k} \rightarrow \mathbb{R}^{k} \cup \{\infty\}$ by
\[
\st_{n,k}(s) :=
\begin{cases}
r & \text{if there exists $r \in \mathbb{R}^{k}$ such that $d_{n,k}(r,s)$ is infinitesimal}, \\
\infty & \text{otherwise.}
\end{cases}
\]

For \(Y\subseteq M^k\), set
\[
\st_{n,k}(Y)
:=
\{\st_{n,k}(y):y\in Y,\ \st_{n,k}(y)\neq\infty\}.
\]
Thus, points \(y\in Y\) for which \(\st_{n,k}(y)=\infty\) make no contribution to \(\st_{n,k}(Y)\).

\begin{remark}
Note that $\st_{n,k}(s)$ might be $\infty$ even if $\st_{n-1,k}(s)\in\mathbb R^k$. This may occur because $\delta^n s$ is unbounded over $\mathbb R$. Indeed, the ordering $<$ does not interact with the derivation $\delta$. There is also a different possibility: every coordinate of $\Jet_n(s)$ may be bounded over $\mathbb R$, while its coordinatewise standard part is not the jet of any element of $\mathbb R^k$. For example, one may have $s$ infinitesimally close to $0$ and $\delta s$ infinitesimally close to $1$. In general, $\st_{n,k}(s)=r\in\mathbb R^k$ only if $\st(\delta^i s_j)=\delta^i r_j$ for every $0\leq i\leq n$ and $1\leq j\leq k$. Nevertheless, if $(\mathbb R,\delta|_{\mathbb R})\preceq(\mathcal M,\delta)$, then $\st_{n,k}$ is definable in $(\mathcal M,\mathbb R,\st,\delta)\models T^\delta_{g,\mathrm{tame}}$.
\end{remark}

Let $A' := (A')^{*} \cap \mathbb{R}^{m}$ and $A := A^{*} \cap \mathbb{R}^{m+k}$. We also assume that for each $a \in A'$, the set $A_{a}$ is closed and bounded with respect to the metric $d_{n,k}$. The next lemma explains how the Hausdorff limit of a sequence from a definable family is related to the standard part of an externally definable set.

\begin{lemma}\label{Hausdorff limits are exactly standard parts}
If $X\in\cl_n(F(A))$, then there exists $a\in(A')^*$ such that $X=\st_{n,k}(A_a^*)$, $A_a^*$ is $\mathbb R$-bounded, and $A_a^*\subseteq U_{d_{n,k}}(X,\varepsilon)$ for every $\varepsilon\in\mathbb R_{>0}$. Conversely, suppose that $(\mathbb R,\delta|_{\mathbb R})\preceq(\mathcal M,\delta)$. If $a\in(A')^*$ is such that $A_a^*$ is $\mathbb R$-bounded and $A_a^*\subseteq U_{d_{n,k}}(\st_{n,k}(A_a^*),\varepsilon)$ for every $\varepsilon\in\mathbb R_{>0}$, then $\st_{n,k}(A_a^*)\in\cl_n(F(A))$.
\end{lemma}

\begin{proof}
``$\Rightarrow$'': Suppose that $X \in \cl_{n}(F(A))$. For each $p \in \mathbb{N}$, let $X_{p} \subseteq X$ be a $2^{-(p+1)}$-$d_{n,k}$-dense finite subset of $X$, i.e., $X \subseteq U_{d_{n,k}}(X_{p},2^{-(p+1)})$. We may assume $X_{p} \subseteq X_{p+1}$ by replacing $X_{p}$ with $\bigcup_{i=0}^{p} X_{i}$. Then $D_{n,k}(X,X_{p}) \leq 2^{-(p+1)}$. Set
\[
C_{p} := \{a \in A' \mid D_{n,k}(A_{a},X_{p}) < 2^{-p}\},
\]
an $\mathcal{L}^{\delta}(X_{p})$-definable subset of $A'$. Note that $C_{p} \neq \varnothing$ since $X \in \cl_{n}(F(A))$. Moreover, $C_{p+1} \subseteq C_{p}$ because if $a \in C_{p+1}$ then $D_{n,k}(X_{p+1},A_{a})<2^{-(p+1)}$ which implies that
\[
D_{n,k}(X_{p},A_{a}) \leq D_{n,k}(X_{p},X_{p+1}) + D_{n,k}(X_{p+1},A_{a}) < 2^{-(p+1)} + 2^{-(p+1)} = 2^{-p}.
\]
Let $\varphi_{p}(x)$ be an $\mathcal{L}^{\delta}(X_{p})$-formula defining $C_{p}$. Then $\Phi(x) := \{\varphi_{p}(x) \mid p \in \mathbb{N}\}$ is a countable set of finitely satisfiable formulas, hence realizable by compactness. By $\aleph_{1}$-saturation of $\mathcal{M}$, there exists $a^{*} \in (A')^{*}$ realizing $\Phi(x)$. Let $Y := A^{*}_{a^{*}}$. For every $p\in\mathbb N$, we have $D_{n,k}(Y,X_p)<2^{-p}$. In particular, since $X_0$ is finite, $Y$ is $\mathbb R$-bounded.

We claim that $\st_{n,k}(Y)=X$. Let $y\in Y$ with $\st_{n,k}(y)=x\in\mathbb R^k$. For every $p\in\mathbb N$, there exists $z_p\in X_p\subseteq X$ such that $d_{n,k}(y,z_p)<2^{-p}$. Since $d_{n,k}(x,y)$ is infinitesimal and $X$ is closed, it follows that $x\in X$. Hence $\st_{n,k}(Y)\subseteq X$.

Conversely, fix $x\in X$. For every $p\in\mathbb N$, choose $z_p\in X_p$ such that $d_{n,k}(x,z_p)\leq 2^{-(p+1)}$. Since $D_{n,k}(Y,X_p)<2^{-p}$, there exists $y_p\in Y$ such that $d_{n,k}(z_p,y_p)<2^{-p}$. Thus $d_{n,k}(x,y_p)<2^{-(p+1)}+2^{-p}$. By $\aleph_1$-saturation, there exists $y\in Y$ such that $d_{n,k}(x,y)$ is infinitesimal. Therefore $\st_{n,k}(y)=x$, and hence $X\subseteq\st_{n,k}(Y)$.

Finally, given $\varepsilon\in\mathbb R_{>0}$, choose $p\in\mathbb N$ such that $2^{-p}<\varepsilon$. Since $D_{n,k}(Y,X_p)<2^{-p}$ and $X_p\subseteq X$, we have $Y\subseteq U_{d_{n,k}}(X,\varepsilon)$.

``$\Leftarrow$'': Suppose that $(\mathbb R,\delta|_{\mathbb R})\preceq(\mathcal M,\delta)$, and let $a\in(A')^*$ satisfy the assumptions of the lemma. Set $Y:=A_a^*$ and $X:=\st_{n,k}(Y)$. Since $Y$ is $\mathbb R$-bounded, so is $X$.

We first show that $X$ is closed. Let $x\in\cl_n(X)$. For each $p\in\mathbb N_{>0}$, choose $x_p\in X$ such that $d_{n,k}(x,x_p)<1/(2p)$. Since $x_p\in\st_{n,k}(Y)$, there exists $y_p\in Y$ such that $d_{n,k}(x_p,y_p)$ is infinitesimal, and hence $d_{n,k}(x,y_p)<1/p$. By $\aleph_1$-saturation, there exists $y\in Y$ such that $d_{n,k}(x,y)$ is infinitesimal. Thus $x=\st_{n,k}(y)\in X$.

Fix $r\in\mathbb R_{>0}$, and let $Z$ be a finite $r/4$-$d_{n,k}$-net in $X$. By the hypothesis on $Y$, we have $Y\subseteq U_{d_{n,k}}(X,r/4)$, and therefore $Y\subseteq U_{d_{n,k}}(Z,r/2)$. Conversely, since $Z\subseteq X=\st_{n,k}(Y)$, every point of $Z$ is infinitesimally close to a point of $Y$. Hence $D_{n,k}(Y,Z)<r/2$.

The statement that there exists $b\in(A')^*$ such that $D_{n,k}(A_b^*,Z)<r/2$ is $\mathcal L^\delta(Z)$-definable and is witnessed by $a$. By elementarity, there exists $c\in A'$ such that $D_{n,k}(A_c,Z)<r/2$. Since $D_{n,k}(X,Z)\leq r/4$, the triangle inequality gives $D_{n,k}(X,A_c)<r$. As $r$ was arbitrary, $X\in\cl_n(F(A))$.
\end{proof}

The preceding lemma identifies Hausdorff limits with the standard parts of the external fibers satisfying the stated boundedness and approximation conditions. We next show that such standard parts are definable.

\begin{lemma} \label{Standard parts are definable in reals}
If $X \subseteq M^{k}$ is $\mathbb{R}$-bounded and definable in the $\mathcal{L}^{\delta}$-structure $(\mathcal{M},\delta)$, then $\st_{n,k}(X) \subseteq \mathbb{R}^{k}$ is definable in the $\mathcal{L}^{\delta}$-structure $(\mathbb{R},\delta|_{\mathbb{R}})$.
\end{lemma}

\begin{proof}
Let $\overline{u}=(u_{i,j})_{1\leq i\leq k,\,0\leq j\leq n}$ and $\overline{v}=(v_{i,j})_{1\leq i\leq k,\,0\leq j\leq n}$, and consider the $\mathcal{L}^{\delta}$-definable set
\[
Y:=\left\{(\overline{u},\overline{v})\in M^{2k(n+1)}:
\exists\overline{c}\in X\
\bigwedge_{i=1}^{k}\bigwedge_{j=0}^{n}
u_{i,j}<\delta^{j}c_{i}<v_{i,j}
\right\}.
\]
By Corollary \ref{Marker-Steinhorn for Derivation Geometric}, the trace $Y\cap\mathbb{R}^{2k(n+1)}$ is definable in $(\mathbb{R},\delta|_{\mathbb{R}})$. Let $\varphi(\overline{u},\overline{v})$ be an $\mathcal{L}^{\delta}(\mathbb{R})$-formula defining this trace.

We claim that for every $\overline{d}\in\mathbb{R}^{k}$, we have $\overline{d}\in\st_{n,k}(X)$ if and only if
\[
(\mathbb{R},\delta|_{\mathbb{R}})\models
\forall\overline{u}\,\forall\overline{v}\,
\left[
\left(
\bigwedge_{i=1}^{k}\bigwedge_{j=0}^{n}
u_{i,j}<\delta^{j}d_{i}<v_{i,j}
\right)
\rightarrow
\varphi(\overline{u},\overline{v})
\right].
\]
Indeed, if $\overline{d}=\st_{n,k}(\overline{c})$ for some $\overline{c}\in X$, then $\Jet_{n}(\overline{c})$ belongs to every box over $\mathbb{R}$ containing $\Jet_{n}(\overline{d})$, so the displayed formula holds.

Conversely, suppose that the displayed formula holds. For every $p\in\mathbb{N}_{>0}$, apply it to the box whose endpoints in the $(i,j)$-coordinate are $\delta^{j}d_i-1/p$ and $\delta^{j}d_i+1/p$. We obtain $\overline{c}_{p}\in X$ such that
\[
|\delta^{j}c_{p,i}-\delta^{j}d_i|<1/p
\]
for every $1\leq i\leq k$ and $0\leq j\leq n$. Hence the countable type
\[
\left\{\overline{x}\in X\right\}
\cup
\left\{
|\delta^{j}x_i-\delta^{j}d_i|<1/p:
p\in\mathbb{N}_{>0},\ 1\leq i\leq k,\ 0\leq j\leq n
\right\}
\]
is finitely satisfiable. By $\aleph_1$-saturation, it is realized by some $\overline{c}\in X$. Then $d_{n,k}(\overline{c},\overline{d})$ is infinitesimal, and therefore $\overline{d}\in\st_{n,k}(X)$. Thus $\st_{n,k}(X)$ is definable in $(\mathbb{R},\delta|_{\mathbb{R}})$.
\end{proof}

\begin{theorem} \label{Hausdorff limits are definable}
Any Hausdorff limit of a sequence from a definable family in $(\mathbb{R},\delta|_{\mathbb{R}})$ is definable in $(\mathbb{R},\delta|_{\mathbb{R}})$.
\end{theorem}

\begin{proof}
This follows directly from Lemmas \ref{Hausdorff limits are exactly standard parts} and \ref{Standard parts are definable in reals}.
\end{proof}

Finally, assuming that $(\mathbb{R},\delta|_{\mathbb{R}})$ is an elementary substructure of $(\mathcal{M},\delta)$, the entire set of Hausdorff limits $\cl_{n}(F(A))$ of a definable family $F(A)$ itself forms a definable family.

\begin{theorem} \label{Hausdorff Limits form a Definable Family}
$\cl_{n}(F(A))$ is definable in $(\mathbb{R},\delta|_{\mathbb{R}})$.
\end{theorem}

\begin{proof}
By either Proposition~\ref{Stable Embedding of T Tame delta G} or Theorem~\ref{Hausdorff limits are definable}, every $X\in\cl_n(F(A))$ is definable in $(\mathbb R,\delta|_{\mathbb R})$. Let $\varphi(\overline{x},\overline{y})$ be the $\mathcal L^\delta(\varnothing)$-formula defining $F(A)$, where $\overline{x}$ is the tuple of parameter variables. Since $\st_{n,k}$ is definable in $(\mathcal M,\mathbb R,\st,\delta)$, let $\rho(\overline{x},\overline{y})$ be an $\mathcal L^\delta_{\mathrm{tame}}(\varnothing)$-formula expressing that $\overline{y}\in\st_{n,k}(A_{\overline{x}}^*)$.

Let $\chi(\overline{x})$ be the $\mathcal L^\delta_{\mathrm{tame}}(\varnothing)$-formula
\[
\begin{aligned}
&\exists r\Bigl(U(r)\wedge r>0\wedge
\forall\overline{w}\bigl(\varphi(\overline{x},\overline{w})
\rightarrow d_{n,k}(\overline{w},0)<r\bigr)\Bigr)\\
&\quad\wedge
\forall\varepsilon\Bigl((U(\varepsilon)\wedge\varepsilon>0)\rightarrow
\forall\overline{w}\bigl(\varphi(\overline{x},\overline{w})\rightarrow
\exists\overline{y}\bigl(U\overline{y}\wedge
\rho(\overline{x},\overline{y})\wedge
d_{n,k}(\overline{w},\overline{y})<\varepsilon\bigr)\bigr)\Bigr).
\end{aligned}
\]
Thus $\chi(\overline{a})$ says that $A_{\overline{a}}^*$ is $\mathbb R$-bounded and is contained in every standard neighbourhood of $\st_{n,k}(A_{\overline{a}}^*)$.

Consider the set of $\mathcal L^\delta_{\mathrm{tame}}(\varnothing)$-formulas
\[
\pi(\overline{x})
:=
\left\{
\forall\overline{z}\,\exists\overline{y}\,
\left[
U\overline{z}\rightarrow
\left(
U\overline{y}\wedge
\neg\bigl(\psi(\overline{y},\overline{z})
\leftrightarrow\rho(\overline{x},\overline{y})\bigr)
\right)
\right]
\;\middle|\;
\psi\text{ is an }\mathcal L^\delta(\varnothing)\text{-formula}
\right\}.
\]
This set is not realizable. Otherwise, for any realization $\overline{a}\in M^m$ of $\pi(\overline{x})$, the set $\st_{n,k}(A_{\overline{a}}^*)\subseteq\mathbb R^k$ would not be definable in $(\mathbb R,\delta|_{\mathbb R})$, contradicting Proposition~\ref{Stable Embedding of T Tame delta G}. Thus, there exists a finite set $\Delta$ of $\mathcal L^\delta(\varnothing)$-formulas such that
\[
(\mathcal M,\mathbb R,\st,\delta)\models
\forall\overline{x}
\bigvee_{\psi\in\Delta}
\exists\overline{z}\,\forall\overline{y}
\left[
U\overline{z}\wedge
\left(
U\overline{y}\rightarrow
\bigl(\psi(\overline{y},\overline{z})
\leftrightarrow\rho(\overline{x},\overline{y})\bigr)
\right)
\right].
\]

Therefore, all sets defined by $\rho(\overline{a},\overline{y})$ on $\mathbb R^k$ are defined by one of the formulas in $\Delta$ with parameters from $\mathbb R$. By a standard coding argument, see \cite[Lemma~2.5]{Guingona2010}, we may assume that there is a single $\mathcal L^\delta(\varnothing)$-formula $\psi(\overline{y},\overline{z})$ which uniformly defines all such sets. Thus, for every $\overline{a}\in M^m$, there exists $\overline{c}\in\mathbb R^{|\overline{z}|}$ such that
\[
\{\overline{b}\in\mathbb R^k:
(\mathbb R,\delta|_{\mathbb R})\models
\psi(\overline{b},\overline{c})\}
=
\st_{n,k}(A_{\overline{a}}^*).
\]

Finally, let $P$ be the set of all $\overline{c}\in\mathbb R^{|\overline{z}|}$ for which there exists $\overline{a}\in(A')^*$ satisfying $\chi(\overline{a})$ and
\[
\{\overline{b}\in\mathbb R^k:
(\mathbb R,\delta|_{\mathbb R})\models
\psi(\overline{b},\overline{c})\}
=
\st_{n,k}(A_{\overline{a}}^*).
\]
The set $P$ is definable in the $\mathcal L^\delta_{\mathrm{tame}}$-structure $(\mathcal M,\mathbb R,\st,\delta)$. By Proposition~\ref{Stable Embedding of T Tame delta G}, it is definable in $(\mathbb R,\delta|_{\mathbb R})$. Lemma~\ref{Hausdorff limits are exactly standard parts} shows that the family defined by $\psi(\overline{y},\overline{c})$, with $\overline{c}\in P$, is exactly $\cl_n(F(A))$. Hence $\cl_n(F(A))$ is definable in $(\mathbb R,\delta|_{\mathbb R})$.
\end{proof}

%% file: Acknowledgement.tex
\subsection*{Acknowledgment}

The author would like to thank Antongiulio Fornasiero and Marcus Tressl for their helpful discussions and valuable suggestions.